\documentclass[11pt]{amsart}
\usepackage{lmodern}
\usepackage[T1]{fontenc}
\usepackage{microtype}
\usepackage{amssymb}
\usepackage{amsthm}
\usepackage{amscd}
\usepackage{amsmath}
\usepackage{hyperref}
\usepackage{cite}

\usepackage{ebproof}

\usepackage{tikz}
\usetikzlibrary{tikzmark}
\usetikzlibrary{decorations.pathreplacing}
\usetikzlibrary{snakes}
\usetikzlibrary{calc}

\newtheorem{theorem}{Theorem}[section]
\newtheorem{lemma}[theorem]{Lemma}

\theoremstyle{definition}
\newtheorem{definition}[theorem]{Definition}

\def \dom{\operatorname{dom}}

\makeatletter

\def\dotminussym#1#2{%
  \setbox0=\hbox{$\m@th#1-$}%
  \kern.5\wd0%
  \hbox to 0pt{\hss\hbox{$\m@th#1-$}\hss}%
  \raise.6\ht0\hbox to 0pt{\hss$\m@th#1.$\hss}%
  \kern.5\wd0}

\mathchardef\mhyphen="2D

\allowdisplaybreaks[2]

\begin{document}

\title{Epsilon Substitution for $ID_1$ via Cut-Elimination}
\author{Henry Towsner}
\date{\today}
\address {Department of Mathematics, University of Pennsylvania, 209 South 33rd Street, Philadelphia, PA 19104-6395, USA}
\email{htowsner@math.upenn.edu}
\urladdr{\url{http://www.math.upenn.edu/~htowsner}}
  \dedicatory{in memory of Grigori Mints}

\begin{abstract}
The $\epsilon$-substitution method is a technique for giving consistency proofs for theories of arithmetic.  We use this technique to give a proof of the consistency of the impredicative theory $ID_1$ using a variant of the cut-elimination formalism introduced by Mints.
\end{abstract}

\maketitle

\section{Introduction}

The $\epsilon$-substitution method was introduced by Hilbert (see \cite{HilbertBernays1970}) in an early attempt to prove the consistency of arithmetic.  The usual quantified formulas $\exists x\ \phi(x,\vec t)$ are replaced by Skolemized formulas $\phi(c_{\exists x\,\phi}(\vec t),\vec t)$ (called $\epsilon$-terms because they have often been written with a quantifier $\phi(\epsilon x\,\phi(x,\vec t),\vec t)$).  An $\epsilon$-substitution is simply a selection of values for these Skolem functions.  By replacing each term with its selected value, one transforms a proof into a sequence of numeric calculations.

In particular, a proof of a $\Sigma_1$ statement is validated if we can find any substitution making the $\Sigma_1$ conclusion true (since we have then found an actual witness to the $\Sigma_1$ statement).  On the other hand, if a given substitution fails to validate the conclusion, it must contain an erroneous step: some \emph{critical formula} $\phi(u,\vec t)\rightarrow\phi(c_{\exists x\,\phi}(\vec t),\vec t)$ in the proof must be false.  This gives us a way to make a small improvement in our substitution, by adding the information that we can interpret $c_{\exists x\,\phi}(\vec t)$ as (the value of) $u$; the sequence of substitutions resulting from this process of successive improvements is the \emph{$H$-process}.  Hilbert proposed that if one could prove that the $H$-process eventually terminates, one has shown the consistency of arithmetic.  After Gentzen had proved the consistency of arithmetic using cut-elimination \cite{MR1513060}, Ackermann proved the termination of this process for Peano arithmetic \cite{Ackermann1940} (using, of course, transfinite induction up to $\epsilon_0$).

Work by many people has extended variations of Gentzen's techniques to theories with the strength of $\Pi^1_1$-comprehension \cite{Takeuti1967,MR644179} and then to stronger theories \cite{MR655036,MR881218,AraiCutMahlo,AraiCutWC,RathjenMahlo}, culminating in systems at the strength of $\Pi^1_2$-comprehension \cite{RathjenStability,RathjenParameterFree,MR2883366,RathjenClaim,AraiClaim}.

These proofs became increasingly complicated (notably, the analyses of $\Pi^1_2$-comprehension remain unpublished twenty years after they were announced), leading Mints to advocate the $\epsilon$-substitution method as an alternate approach.  In particular, Mints \cite{Mints1994} developed a combination technique, in which one shows that the $H$-process terminates by constructing a different sequent calculus for substitutions and applying cut-elimination to this process.  This method was extended to more powerful theories, including elementary analysis \cite{MintsTupailoBuchholz1996}, ramified analysis \cite{MintsTupailo1999}, and the hyperarithmetical hierarchy \cite{Arai2002}.  

All these systems, however, are predicative, and impredicativity has always been a significant hurdle for proof-theoretic methods \cite{MR2457679}.  Arai \cite{Arai2003} created a version of the $\epsilon$-substitution method for an impredicative theory (a certain theory of inductive definitions) and proved its termination using a proof closer to Ackermann's proof, later extending this to a stronger theory of inductive definitions \cite{MR2275854}.

Mints continued to look for a proof of the termination of the $H$-process for impredicative theories.  This search focused on the theory $ID_1$, which adds an inductive predicate to Peano arithmetic.  After Mints (and, starting in 2003, the author) struggled for several years to find such a proof, Mints proposed a modified, non-deterministic $H$-process for $ID_1$ was able to prove that it terminated, first by non-effective means \cite{mints_centenary} and then by effective ones \cite{MR3204984}.  This was still not completely satisfactory: the non-determinism cannot be necessary, since Arai \cite{Arai2003} was able to prove termination of the original $H$-process.


This paper address that gap: we formulate the ordinary $H$-process for $ID_1$---essentially the same process as in Arai's work, adapted to avoid ordinal notations---and prove that it terminates via a cut-elimination proof in the style of \cite{MintsTupailoBuchholz1996}.  In addition to those ideas and inspiration from Arai's version \cite{Arai2003}, we use adaptations to the cut-elimination techniques introduced in \cite{MR2139689}.

\subsection{Dedication}

This paper is dedicated to Grisha Mints.  Without Grisha's guidance, I would never have become a mathematician, and for many years I was fortunate to have his continued support and advice.

\section{$ID_1$}

$ID_1$ is the theory of Peano arithmetic augmented by a predicate standing for an inductively defined set.  Formally, we fix a formula $A(x,X)$, where $X$ is a second order variable which appears positively in $A$, and add a new predicate $I$ to the language.  (For clarity, we typically ``abbreviate'' the formula $It$ by $t\in I$.)  We add an axiom and an axiom scheme:
\begin{itemize}
\item Inductive definition: $\forall x(x\in I\leftrightarrow A(x,I))$, and
\item Inductive closure: $\forall x(A(x,\phi)\rightarrow\phi(x))\rightarrow \forall x(x\in I\rightarrow\phi(x))$.
\end{itemize}

 It is known that restricting $A$ to $\Pi_1$ formulas does not weaken the system \cite{EIAS}, and since Peano arithmetic can encode pairing, we may further assume that $A(x,X)$ has the form $\forall y\, B(y,x,X)$ where $B$ is quantifier-free.

\section{Skolem Functions and $\epsilon$-substitution}

\subsection{The Skolemized Language}

We work in a Skolemized version of the language of $ID_1$.  That is, taking $\mathcal{L}_0$ to be the language of $ID_1$, we work in a language $\mathcal{L}_\epsilon$ which replaces quantifiers with countably many new function symbols intended to be witnesses to existential quantifiers.

\begin{definition}
  We say a formula $\phi(x,\vec y)$ with distinguished variable $x$ is \emph{simple} if $\phi(x,\vec y)$ contains no closed terms.

A \emph{Skolem term} (in a language $\mathcal{L}$ extending $\mathcal{L}_0$) is a term of the form $c(\vec t)$ where $c$ is a function symbol not present in $\mathcal{L}_0$.

  Whenever $\mathcal{L}$ is a language extending $\mathcal{L}_0$, define $\mathcal{L}'$ by adding, for each formula $\exists x\,\phi(x,\vec y)$ such that $\phi(x,\vec y)$ is quantifier-free and simple, a $|\vec y|$-ary Skolem function $c_{\exists x\,\phi(x,\vec y)}$.

  Let $\mathcal{L}_0$ be the language of $ID_1$.  Given $\mathcal{L}_n$, we define $\mathcal{L}_{n+1}=\mathcal{L}'_n$.  We set $\mathcal{L}_{\omega}=\bigcup\mathcal{L}_n$ and take $\mathcal{L}_\epsilon$ to be the quantifier-free part of $\mathcal{L}_{\omega}$.

The \emph{simple rank}, $rk_s(t)$, of an expression is the smallest $n$ such that the expression appears in $\mathcal{L}_n$.
\end{definition}

Having removed quantifiers from our language, we reintroduce them as abbreviations for certain formulas with Skolem functions; this will be unambiguous because all quantifiers appearing in the rest of this paper will be abbreviations of this kind.
\begin{definition}
  $\exists x\,\phi(x,\vec t)$ is an abbreviation for $\phi(c_{\exists x\,\phi}(\vec t),\vec t)$ and $\forall x\,\phi(x,\vec t)$ is an abbreviation for $\phi(c_{\exists x\,\neg\phi}(\vec t),\vec t)$.
\end{definition}

In particular, we have the distinguished formula $A(y,X)=\forall x\,B(x,y,X)$; since $B$ is quantifier-free, $B$ is a formula of $\mathcal{L}_\epsilon$, and we may reinterpret $A(y,X)$ as the appropriate formula
\[B(c_{\exists x\,\neg B(x,y,X)}(y),y,X).\]

\begin{definition}
  The only rule of $ID_1^\epsilon$ is modus ponens.  The axioms are:
  \begin{enumerate}
  \item All propositional tautologies,
  \item All substitution instances of quantifier free defining axioms for all predicates in the language of Peano arithmetic,
  \item Equality axioms $t=t$, $s=t\rightarrow\phi(s)\rightarrow\phi(t)$ for all terms $s,t$,
  \item $\neg St=0$ and $Ss=St\rightarrow s=t$,
  \item The \emph{critical formulas}:
    \begin{itemize}
    \item \textbf{Pred}: $\neg s=0\rightarrow \exists x\, s=Sx$,
    \item \textbf{Epsilon}: $\phi(t)\rightarrow\exists x\, \phi(x)$,
    \item \textbf{Induction}: $\phi(0)\wedge\neg\phi(t)\rightarrow\exists x\,(\phi(x)\wedge \neg\phi(Sx))$,
    \item \textbf{Inductive Definition}: $A(t,I)\rightarrow t\in I$,
    \item \textbf{Closure}: $\forall x(A(x,\phi)\rightarrow\phi(x))\rightarrow\forall x(x\in I\rightarrow\phi(x))$.
    \end{itemize}
  \end{enumerate}
\end{definition}

\section{$\epsilon$-Substitutions}

\begin{definition}
  A \emph{canonical expression} is either:
  \begin{itemize}
  \item a term of the form $c(\vec t)$ where $c$ is a Skolem function and each $t_i$ is a numeral, or
  \item a formula of the form $t\in I$ where $t$ is a numeral.
  \end{itemize}

We abbreviate the rank $\Omega$ canonical term $c_{\exists x\,\neg B(x,y,I)}(n)$ by $c_n$.
\end{definition}

\begin{definition}
  The \emph{rank}, $rk(c)$ of a function symbol $c=c_{\exists x\,\phi(x,\vec y)}$ is given as follows:
  \begin{itemize}
  \item If $\phi$ does not contain $I$, $rk(c_{\exists x\,\phi(x,\vec y)})=rk_s(c_{\exists x\,\phi(x,\vec y)}(\vec 0))$,
  \item If $\phi$ does contain $I$ and is not $B(x,y,I)$, $rk(c_{\exists x\,\phi(x,\vec y)})=\Omega+rk_s(c_{\exists x\,\phi(\vec y)}(\vec 0))+1$,
  \item $rk(c_{\exists x\, \neg B(x,y,I)})=\Omega$.\footnote{Grisha would not have approved of this definition.  Many, many drafts ago, he warned me that an inductive definition should not require this sort of exceptional case.  This was, as always, good advice: in every other case where I have been tempted to do such a thing, finding the correct definition not only mooted the need to create an exception, but simplified other parts of the proof as well.  In this case, however, I have been unable to find the correct definition, and the current one seems to work well enough.}
  \end{itemize}

The rank of the formula $t\in I$ is $\Omega$.
\end{definition}
When $e$ is a canonical term $c(\vec t)$, we often refer to the rank $rk(e)$ to mean $rk(c)$.

\begin{definition}
  An \emph{$\epsilon$-substitution} is a function $S$ such that:
  \begin{itemize}
  \item The domain of $S$ is a set of canonical expressions,
  \item If $e\in\dom(S)$ is a term then $S(e)$ is either a positive numeral or the symbol $?$,
  \item If $e\in\dom(S)$ is a formula then $S(e)$ is either $\top$ or the symbol $?$.
  \end{itemize}
An $\epsilon$-substitution is \emph{total} if its domain is the set of all canonical terms.
\end{definition}

\begin{definition}
  If $S$ is an $\epsilon$-substitution and $\bowtie\in\{<,\leq,=,\geq,>\}$, we define
\[S_{\bowtie r}=\{(e,u)\in S\mid rk(e)\bowtie r\}.\]
\end{definition}

\begin{definition}
  The \emph{standard extension} $\overline{S}$ is given by:
\[\overline{S}=S\cup\{(e,?)\mid e\not\in\dom(S)\}.\]
\end{definition}

\begin{definition}
  If $t$ is a term, we define the \emph{reduction} of $t$ with respect to $S$, $|t|_S$, by induction on $t$:
  \begin{itemize}
  \item If $t$ is a canonical Skolem term in the domain of $S$ and $S(t)={?}$ then $|t|_S=0$,
  \item If $t$ is a canonical Skolem term in the domain of $S$ and $S(t)\neq {?}$ then $|t|_S=S(t)$,
  \item If $t$ is a canonical Skolem term not in the domain of $S$ then $|t|_S=t$,
  \item If $t$ is a non-canonical Skolem term of the form $c(\vec s)$ and for some $i$, $|s_i|_S$ is not a numeral then $|t|_S=c(|\vec s|_S)$,
  \item If $t$ is a non-canonical Skolem term of the form $c(\vec s)$ and for every $i$, $|s_i|_S$ is a numeral then $|t|_S=|c(|\vec s|_S)|_S$,
  \item If $t$ is a term of the form $f(\vec s)$ for some function symbol $f$ of Peano arithmetic, and for some $i$, $|s_i|_S$ is not a numeral then $|t|_S$ is the term $f(|\vec s|_S)$,
  \item If $t$ is a term of the form $f(\vec s)$ for some function symbol $f$ of Peano arithmetic, and for every $i$, $|s_i|_S$ is a numeral then $|t|_S=\hat f(|\vec s|)$ where $\hat f$ is the interpretation of $f$ in the standard model.
  \end{itemize}

If $\phi$ is a formula, we define $|\phi|_S$ inductively by:
\begin{itemize}
\item $|R\vec t|_S=R|\vec t|_S$ for any relation symbol $R$,
\item $|\phi\wedge\psi|_S=|\phi|_S\wedge|\psi|_S$,
\item $|\phi\vee\psi|_S=|\phi|_S\vee|\psi|_S$,
\item $|\neg\phi|_S=\neg|\phi|_S$,
\item If $|t|_S$ is a numeral, $|t|_S\in I$ is in $\dom(S)$, and $S(|t|\in I)=\top$ then $|t\in I|_S=\top$,
\item If $|t|_S$ is a numeral, $|t|_S\in I$ is in $\dom(S)$, and $S(|t|\in I)={?}$ then $|t\in I|_S=\bot$,
\item If $|t|_S$ is not a numeral or $|t|_S\in I$ is not in $\dom(S)$ then $|t\in I|_S=|t|_S\in I$.
\end{itemize}

We say $S\vDash\phi$ if $|\phi|_S$ is sentence in the language of Peano arithmetic which is true in the standard model.
\end{definition}

Note that it is possible to have $S\not\vDash\phi$ and $S\not\vDash\neg\phi$.

\begin{definition}
  $S$ \emph{decides} $\phi$ if $S\vDash\phi$ or $S\vDash\neg\phi$.
\end{definition}
It is easy to check that if $S$ is total then $S$ decides all sentences.

\begin{lemma}\label{thm:low_rank_equiv}
  \begin{enumerate}
  \item If all expressions  in $t$ have rank $<r$ and $S_{<r}=S'_{<r}$ then $|t|_S=|t|_{S'}$.
  \item If all expressions in $\phi$ have rank $<r$ and $S_{<r}=S'_{<r}$ then $S\vDash\phi$ iff $S'\vDash\phi$.
  \end{enumerate}
\end{lemma}
\begin{proof}
The first part is a straightforward induction on $t$.  The second part is a straightforward induction on $\phi$, using the first part.
\end{proof}

We would like our substitutions to be correct; for instance, if $(c_{\exists x\,\phi(x,\vec y)}(\vec t),n)\in S$, it should be the case that $\overline{S}\vDash\phi(n,\vec t)$.  To make this precise, we identify, for each $(e,u)$, a formula $F(e,u)$, with the expectation that $F(e,u)$ should be true if $(e,u)$ is to be put in $S$.

\begin{definition}
\[F(e,u)=\left\{\begin{array}{ll}
\top&\text{if }u={?}\\
\phi(u,\vec t)\wedge\bigwedge_{v<u}\neg\phi(v,\vec t)&\text{if }e=c_{\exists x\,\phi(x,\vec y)}(\vec t)\text{ and }u\neq {?}\\
A(t,I)&\text{if }e=t\in I\text{ and }u\neq {?}
\end{array}\right.\]

We say $S$ is \emph{correct} if for every $(e,u)\in S$, $\overline{S}\vDash F(e,u)$.  We say $S$ is \emph{computationally consistent} (cc) if whenever $(e,u)\in S$, $S\not\vDash\neg F(e,u)$, and if there is not an $n$ such that $(c_n,u)\in S$ with $u\neq {?}$ and also $(n\in I,\top)\in S$.  We say $S$ is \emph{computationally inconsistent} (ci) if $S$ is not cc---if there is an $(e,u)\in S$ such that $S\vDash \neg F(e,u)$, or if for some $n, u\neq {?}$, $(c_n,u),(n\in I,\top)\in S$.

$S$ is \emph{computing} if for every $(e,u)\in S$, $S$ decides $F(e,u)$.
\end{definition}

\begin{lemma}\label{thm:low_rank_correct}
  If $rk(e)=r\neq\Omega$, $S\vDash F(e,u)$, and $S_{<r}=S'_{<r}$ then $S'\vDash F(e,u)$.
\end{lemma}
\begin{proof}
When $r\neq\Omega$, all canonical expressions occurring in $F(e,u)$ have rank $<r$, so the claim follows from Lemma \ref{thm:low_rank_equiv}.
\end{proof}

Note that there are two types of canonical expressions of rank $\Omega$, formulas of the form $n\in I$ and the terms $c_n$ (i.e. terms of the form  $c_{\exists x\,\neg B(x,y,I)}(n)$). 
\begin{definition}
Expressions of the form $n\in I$ are \emph{positive rank $\Omega$ expressions}.  We write 
\[S_{=\Omega}^{+,\mathrm{form}}=\{(e,\top)\in S\mid rk(e)=\Omega\text{ and }e\text{ is a formula}\},\]
\[S_{=\Omega}^{+,\mathrm{term}}=\{(e,?)\in S\mid rk(e)=\Omega\text{ and }e\text{ is a term}\},\]
and
\[S_{=\Omega}^+=S_{=\Omega}^{+,\mathrm{form}}\cup S_{=\Omega}^{+,\mathrm{term}}.\]

Expressions of the form $c_n$ are \emph{negative rank $\Omega$ expressions}.  We write 
\[S_{=\Omega}^{-,\mathrm{term}}=\{(e,v)\in S\mid rk(e)=\Omega\text{ and } e\text{ is a term}\},\]
\[S_{=\Omega}^{-,\mathrm{form}}=\{(e,?)\in S\mid rk(e)=\Omega\text{ and }e\text{ is a formula}\},\]
and
\[S_{=\Omega}^-=S_{=\Omega}^{-,\mathrm{term}}\cup S_{=\Omega}^{-,\mathrm{form}}.\] 
\end{definition}
 The fact that positive and negative expressions can be separated by distinguishing formulas from terms is an incidental feature of our formalism; the important fact is that $I$ appears positively in $F(e,u)$ when $e$ is a positive expression and negatively when $e$ is negative.  Rank $\Omega$ terms have cyclical dependencies---$F(n\in I,\top)$ is the formula $A(n,I)$, which contains $c_{\exists x\,\neg B(x,y,I)}(n)$, while $F(c_{\exists x\,\neg B(x,y,I)},m)(n)$ is $\neg B(m,n,I)$, which may contain various expressions $t\in I$.  The analog to Lemma \ref{thm:low_rank_correct} depends on the fact that $I$ appears positively in $B(x,y,I)$, so negatively in $\neg B(x,y,I)$, and so again positively in $A(y,I)$.
\begin{lemma}\label{thm:low_rank_correct_Omega}
Suppose $rk(e)=\Omega$, $\overline{S}\vDash F(e,u)$, and $S_{<\Omega}=S'_{<\Omega}$.
  \begin{itemize}
  \item If $e$ is a positive rank $\Omega$ expression and $(S')^-_{<\Omega}\subseteq S^-_{<\Omega}$ then $\overline{S'}\vDash F(e,u)$,
  \item If $e$ is a negative rank $\Omega$ expression and $(S')^+_{<\Omega}\subseteq S_{<\Omega}^+$ then $\overline{S'}\vDash F(e,u)$.
  \end{itemize}
\end{lemma}
\begin{proof}
  In the first case, $F(e,u)$ is $A(n,I)$ for some $n$, so if $\overline{S'}\not\vDash F(e,u)$, it must be because $(c_{\exists x\,\neg B(x,y,I)}(n),u)\in S'$ with some $u\neq{?}$.  Therefore also $(c_{\exists x\,\neg B(x,y,I)}(n),u)\in S$, so $\overline{S}\not\vDash F(e,u)$.

  In the second case, $F(e,u)$ is $\neg B(u,n,I)$ for some $u,n$.  The only canonical expressions appearing in $\neg B(u,n,I)$ are expressions of the form $t\in I$, all of which appear negatively.  The claim follows immediately using the monotonicity assumption.
\end{proof}



The main point of these definitions is the following:
\begin{theorem}
  Suppose that for every proof of a formula $\phi$ in $ID_1$, there is a computing correct $\epsilon$-substitution $S$ such that $S\vDash\phi$.  Then $ID_1$ is $1$-consistent.
\end{theorem}
\begin{proof}
  Suppose $ID_1\vdash\exists x\,\phi(x,\vec t)$ where $\phi$ is quantifier free and the $\vec t$ are numerals.  Then $ID_1^\epsilon\vdash\phi(c_{\exists x\,\phi}(\vec t),\vec t)$.  By assumption there is an $S$ so that $S\vDash\phi(c_{\exists x\,\phi}(\vec t),\vec t)$, and so $\phi(c_{\exists x\,\phi}(\vec t),\vec t)$ is a true quantifier-free formula, so $\exists x\,\phi(x,\vec t)$ is a true $\Sigma_1$ formula witnessed by the number $c_{\exists x\,\phi}(\vec t)$.
\end{proof}

The following properties are standard \cite{Ackermann1940,MintsTupailoBuchholz1996}:
\begin{lemma}
  \begin{itemize}
  \item If $\sigma$ is an axiom other than a critical formula then $\overline{S}\vDash\sigma$, and
  \item If $S\vDash\phi$ and $S\vDash\phi\rightarrow\psi$ then $S\vDash\psi$.
  \end{itemize}
\end{lemma}

Therefore to prove the $1$-consistency of $ID_1$, it suffices to show that for every proof in $ID_1$, we can construct $S$ satisfying every critical formula in that proof.

\subsection{History}

Because $ID_1$ is impredicative, it is not enough to track substitution values: when we want to make $n\in I$ true, we also have to track the history which justified that.

\begin{definition}
  A \emph{substitution history} is a pair $P,V$ such that $P$ is a finite sequence of distinct formulas, each of the form $n\in I$ for some $n$, and $V$ is a function whose domain is the set of formulas appearing in $P$ and for each $n\in I$ in $\dom(V)$, $V(n\in I)$ is a finite $\epsilon$-substitution with $(V(n\in I))^-_{=\Omega}=V(n\in I)$.

  $P=\langle n_1\in I,\ldots,n_k\in I\rangle$ is \emph{admissible} if for each $i\leq k$, $|B(0,n_i,I)|_{\overline{\{(n_j\in I,\top)\}_{j<i}}}$ is true.

A \emph{historical substitution} is a triple $(S,P,V)$ where $P,V$ is a substitution history and the sequence $P$ is an ordering of $\dom(S_{=\Omega}^{+,\mathrm{form}})$.
\end{definition}
Recall that $S_{=\Omega}^{+,\mathrm{form}}$ is the set of $n\in I$ such that $(n\in I,\top)\in S$.

We often equate historical substitutions $(S,P,V)$ with just their substitution $S$; for instance, we speak of the domain of $(S,P,V)$, by which we mean the domain of $S$.

When we have a sequence of historical substitutions $(S_1,P_1,V_1),\ldots,(S_k,P_k,V_k)$, the sequences $P_i$ represent the order in which formula $n\in I$ appeared in $S_j$ for $j\leq i$.  When we update $S_k$ to $S_{k+1}$ by adding a formula $n\in I$, we may have to remove terms of the form $(c_m,u)$ which are no longer correct, and we record these formulas in $V_{k+1}(n\in I)$.  At some $k'>k$, we might add $(c_n,u)$ for some $u\neq {?}$, and therefore we should remove the formula $n\in I$ to preserve correctness.  When we do this, we also return the formulas $V_{k+1}(n\in I)$ to $S_{k'+1}$.  If we did not do this, it would be possible for the process we describe below to get caught in a loop.

\subsection{The $H$-Process}

From here on, let $\mathrm{Cr}=\{Cr_0,\ldots,Cr_N\}$ be a fixed finite list of closed critical formulas.

\begin{definition}
  We say $S$ is solving if for each $I\leq N$, $\overline{S}\vDash Cr_I$.
\end{definition}

When $S$ is a correct nonsolving substitution, there are some (and possibly many) critical formulas $Cr_I$ which are not currently satisfied by $\overline{S}$; we wish to modify $S$ to satisfy one of them.  (Inevitably, doing so may cause formulas which were previously satisfied to become unsatisfied; the main work of this paper is showing that the resulting process eventually terminates.)  To do this, we identify, for each unsatisfied $Cr_I$, a pair $(e,v)$ which we could add to $S$ to satisfy that particular formula.  We will ultimately choose to add a pair minimizing $rk(e)$.

\begin{definition}
  Let $(S,P,V)$ be a historical substitution where $S$ is correct and nonsolving.

  For $I\leq N$ such that $\overline{S}\not\vDash Cr_I$, define a term $e^{S,P,V}_I$ depending on the type of the formula $Cr_I$:
  \begin{enumerate}
  \item If $Cr_I$ is a \textbf{Pred} axiom $\neg s=0\rightarrow\exists x\,s=Sx$ then
\[e^{S,P,V}_I=c_{\exists x\,|s|_{\overline{S}}=Sx} \text{ and } v^{S,P,V}_I=|s|_{\overline{S}}-1.\]
\item If $Cr_I$ is an \textbf{Epsilon} axiom $\phi(t)\rightarrow\exists x\, \phi(x,\vec s)$ then
\[e^{S,P,V}_I=c_{\exists x\,|\phi|_{\overline{S}}}(|\vec s|_{\overline {S}}) \text{ and } v^{S,P,V}_I=|t|_{\overline{S}}.\]
\item If $Cr_I$ is an \textbf{Induction} axiom of the form $\phi(0,\vec s)\wedge\neg\phi(t,\vec s)\rightarrow\exists x\,(\phi(x,\vec s)\wedge \neg\phi(Sx,\vec s))$ then
\[e^{S,P,V}_I=c_{\exists x\,(|\phi(x)|_{\overline{S}}\wedge\neg|\phi(Sx)|_{\overline{S}})}(|\vec s|_{\overline{S}})\]
and $v^{S,P,V}_I$ is the least $m<|t|_{\overline{S}}$ such that $\overline{S}\vDash\phi(m,\vec s)$ but $\overline{S}\vDash\neg\phi(m+1,\vec s)$.
\item If $Cr_I$ is an \textbf{Inductive Definition} axiom of the form $A(t,I)\rightarrow t\in I$ then
\[e^{S,P,V}_I=|t|_{\overline{S}}\in I\text{ and }v^{S,P,V}_I=\top.\]
\item If $Cr_I$ is a \textbf{Closure} axiom of the form $\forall x(A(x,\phi(x))\rightarrow\phi(x))\rightarrow\forall x(x\in I\rightarrow\phi(x))$ then let $n\in I$ be the first formula in $P$ such that $\overline{S}\vDash \neg\phi(n)$.  We consider two cases; if $\overline{S}\vDash\neg A(n,\phi(n))$ then we take
\[e^{S,P,V}_I=c_n\text{ and }v^{S,P,V}_I=\overline{S}(c_{\exists x\, |B(x,y,\phi)|_{\overline{S}}}(\vec s,n)).\]
Otherwise we take
\[e^{S,P,V}_I=c_{\exists x\, A(x,|\phi|_{\overline{S}})\wedge\neg|\phi(x)|_{\overline{S}}}(\vec s) \text{ and }v^{S,P,V}_I=n\]
where $\vec s$ is the closed terms appearing in $|\phi|_{\overline{S}}$.
  \end{enumerate}

For each $I$, let $r_I^{S,P,V}=rk(e^{S,P,V}_I)$.
\end{definition}


\begin{definition}
  When $(S,P,V)$ is a historical substitution with $S$ correct and nonsolving, we let $I(S,P,V)$ be the $I$ such that $r_I^{S,P,V}$ is least among $I\leq N$ with $\overline{S}\not\vDash Cr_I$ and $I(S,P,V)$ is least among such $I$.
\end{definition}

We now describe how we update the historical substitution $(S,P,V)$ to a new substitution $H(S,P,V)$.  We abbreviate $I(S,P,V)$ by $I$.  The two cases where $r_I^{S,P,V}=\Omega$ are more complicated (and the essential ones for the extension of the process to $ID_1$) so we describe them separately.

\begin{definition}
Let $(S,P,V)$ be a historical sequent with $S$ correct and nonsolving.  Let $I$ be such that $r_I^{S,P,V}$ is least among those $I$ with $\overline{S}\not\vDash Cr_I$ and $I$ is least among such indices.  We set $Cr(S,P,V)=Cr_I$, $r(S,P,V)=r^{S,P,V}_I$, $e(S,P,V)=e^{S,P,V}_I$, and $v(S,P,V)=v^{S,P,V}_I$.  We call $e(S,P,V)$ the \emph{$H$-expression} and $v(S,P,V)$ the \emph{$H$-value}.

If $r_I^{S,P,V}>\Omega$ then we set
\[H(S,P,V)=((S_{\leq r(S,P,V)}\setminus\{(e^{S,P,V}_I,?)\})\cup\{(e^{S,P,V}_I,v^{S,P,V}_I)\},P,V).\]

If $r_I^{S,P,V}<\Omega$ then we set
\[H(S,P,V)=((S_{\leq r(S,P,V)}\setminus\{(e^{S,P,V}_I,?)\}))\cup\{(e^{S,P,V}_I,v^{S,P,V}_I)\},\emptyset,\emptyset).\]
\end{definition}

We now turn to the case where $r_I^{S,P,V}=\Omega$.  In the positive $\Omega$ case we add a new $n\in I$ and must update $P$ and $V$ accordingly.
\begin{definition}
  If $r(S,P,V)=\Omega$ and $e^{S,P,V}_I$ is the positive $\Omega$ expression $n\in I$, we set
\[H(S,P,V)=(S_{<\Omega}\cup S^+_{=\Omega}\cup\{(e^{S,P,V}_I,\top)\},P{}^\frown\langle n\in I\rangle,V\cup\{(n\in I,S_{=\Omega}^{-})\}).\]
\end{definition}
As in the $r\neq \Omega$ case, we remove all expressions of higher rank.  We also remove negative expressions of rank $\Omega$---note that they may no longer be correct---but we record the expressions we have removed in $V$.

In the negative $\Omega$ case we have potentially invalidated some $(n\in I,\top)$.  We wish to do more than just remove this element; we wish to restore the rank $\Omega$ part of $S$ to the state it was at when we first added $(n\in I,\top)$.  (Note that we do not distinguish whether $Cr_I$ was a \textbf{Closure} axiom of an \textbf{Epsilon} axiom which happens to have a term of rank $\Omega$ as its main term.)
\begin{definition}\label{def:H_Omega_neg}
If $r(S,P,V)=\Omega$ and $e_I^{S,P,V}$ is the negative $\Omega$ expression $c_n$, we define:
\begin{itemize}
\item $P'$ is the longest initial segment of $P$ not containing $n\in I$,
\item $S'$ is those $(c_m,?)\in S$ and those $(m\in I,\top)$ such that $m\in I$ appears in $P'$,
\item if $P'{}^\frown\langle n\in I\rangle\sqsubseteq P$, $V'=V(n\in I)$,
\item if $P'=P$, $V'=S^{-}_{=\Omega}$.
\end{itemize}

We set
\[H(S,P,V)=(S_{<\Omega}\cup S'\cup V'\cup \{(c_n,v^{S,P,V}_I)\},P',V\upharpoonright P').\]
\end{definition}
To interpret this definition, note that adding a negative $\Omega$ expression may invalidate positive $\Omega$ expressions.  In particular, if $(n\in I,\top)$ appears in $S$, it conflicts with $(c_n,v^{S,P,V}_I)$ and must be removed; further, respecting the history $P,V$ requires assuming that any $m\in I$ added after $n\in I$ might depend on $n\in I$ to be justified, so these must be removed as well.  We also remove all $(c_m,?)$ other than those which must remain (because they are part of the justification of an $m\in I$ which comes before $n\in I$).  Finally, when we remove $n\in I$, we should restore all those negative expressions which the addition of $(n\in I,\top)$ caused us to remove, which is the addition of $V'$.

This gives us the \emph{$H$-process}:
\begin{itemize}
\item $(S_0,P_0,V_0)=(\emptyset,\emptyset,\emptyset)$,
\item $(S_{k+1},P_{k+1},V_{k+1})=\left\{\begin{array}{ll}
H(S_k,P_k,V_k)&\text{if }S_k\text{ is nonsolving}\\
(S_k,P_k,V_k)&\text{if }S_k\text{ is solving}
\end{array}\right..$
\end{itemize}

\begin{lemma}
Each $S_k$ is correct.
\end{lemma}
\begin{proof}
By induction on $k$.  $S_0$ is trivially correct.

Suppose $S_k$ is correct.  First, suppose $r(S_k,P_k,V_k)\neq\Omega$.  If $(e,u)\in S_k$ then the result follows from the correctness of $S_k$, the fact that $(S_k)_{<r(S_k)}=(S_{k+1})_{<r(S_k)}$, and Lemma \ref{thm:low_rank_correct}.

If $(e,u)\in S_{k+1}\setminus S_k$ then $e=e(S_k,P_k,V_k)$, $Cr(S_k,P_k,V_k)$ was a \textbf{Pred}, \textbf{Epsilon}, \textbf{Induction} axiom or the second case of a \textbf{Closure} axiom, and the definition of $e(S_k,P_k,V_k)$ and the fact that $(S_k)_{<r(S_k)}=(S_{k+1})_{<r(S_k)}$ ensures that $\overline{S_{k+1}}\vDash F(e,u)$.

So suppose $rk(S_k,P_k,V_k)=\Omega$.  If $(e,u)\in S_{k+1}$ with $rk(e)<\Omega$ then $(e,u)\in S_k$ and the result again follows from the correctness of $S_k$, the fact that $(S_k)_{<r(S_k)}=(S_{k+1})_{<r(S_k)}$, and Lemma \ref{thm:low_rank_correct}.

If $e(S_k,P_k,V_k)$ is a positive $\Omega$ expression then $(S_{k+1})^-_{=\Omega}=\emptyset$, so we need only consider $(S_{k+1})^{+,\mathrm{form}}_{=\Omega}$.  If $(e,u)\in(S_{k+1})^{+,\mathrm{form}}_{=\Omega}$ then $\overline{S_{k+1}}\vDash F(e,u)$ because $(S_{k+1})^-_{=\Omega}=\emptyset$.

Suppose $e(S_k,P_k,V_k)$ is a negative $\Omega$ expression $c_n$.  The case where $(n\in I,\top)\not\in S_k$ is much like the case where the rank is not $\Omega$.  If $(e,u)\in (S_{k+1})^-_{=\Omega}$ then $\overline{S_k}\vDash F(e,u)$, and since $(S_{k+1})^+_{=\Omega}=(S_k)^+_{=\Omega}$ we have $\overline{S_{k+1}}\vDash F(e,u)$ by Lemma \ref{thm:low_rank_correct_Omega}.  If $(e,u)\in (S_{k+1})^+_{=\Omega}$ then either $u={?}$ or $e$ is $m\in I$; since $m\neq n$, $(S_{k+1})^-_{=\Omega}=(S_k)^-_{=\Omega}\cup\{(c_n,v(S_k,P_k,V_k))\}$, and $\overline{S_k}\vDash F(e,u)$, also $\overline{S}_{k+1}\vDash F(e,u)$.

Finally, consider the case where $e(S_k,P_k,V_k)$ is a negative $\Omega$ expression $c_n$ and $(n\in I,\top)\in S_k$.  Let $j<k$ be maximal so that $e(S_j,P_j,V_j)$ is $n\in I$---that is, the step from $S_j$ to $S_{j+1}$ was the step at which $n\in I$ was added.  By a straightforward induction, $P_j$ must be an initial segment of each $P_{j'}$ for $j\leq j'\leq k$, and therefore $P_j=P_{k+1}$.  By the inductive hypothesis, $S_j$ is correct.  Note that $(S_{k+1})_{=\Omega}=(S_j)_{=\Omega}\cup\{(c_n,v(S_k,P_k,V_k))\}\cup S'$ where $S'$ consists of some pairs $(c_m,?)$ where $(m\in I,\top)\in S_j$.  Then for each $(e,u)\in S_{k+1}$, $\overline{S_{k+1}}\vDash F(e,u)$ as in the previous case.
\end{proof}

What remains is showing that the $H$-process eventually terminates: that is, that there is a $k$ with $S_{k+1}=S_k$, since such an $S_k$ is necessarily correct and solving.

\section{Deduction System}

When we defined the $H$-process, we defined it in terms of sequents $(S_k,P_k,V_k)$ where $S_k(e)$ is never ${?}$, because we worked with $\overline{S_k}$ rather than $S_k$.  However below it will be more useful to stay within $S$---in particular, to require that if $(e,?)$ is used in the computation of the $H(S,P,V)$ then $(e,?)$ is actually present in $S$.

\begin{definition}
  If $S$ is correct and nonsolving, we say the \emph{$H$-step applies} to $(S,P,V)$ if all calculations in the computation of $H(S,P,V)$ take place in $S$. 

  If the $H$-step applies to $(S,P,V)$, we say $e$ is \emph{active} at $(S,P,V)$ if $e\in\dom(S)$, $rk(e)=r(S,P,V)$, and either $e\not\in\dom(H(S,P,V))$ or the value of $e$ in $H(S,P,V)$ is different from the value in $S$.
\end{definition}
When we say ``all calculations in the computation of $H(S,P,V)$'', we include that if $e(S,P,V)=c_n$ then $n\in I$ must be in $\dom(S)$, and conversely if $e(S,P,V)=n\in I$ then $c_n$ must be in $\dom(S)$.  We also include that if $(c_n,?)\in S$ and $e(S,P,V)$ is a term of rank $\Omega$ then $n\in I\in\dom(S)$ (so that we know whether or not to include $(c_n,?)\in H(S,P,V)$.

Note that if the $H$-step applies to $(S,P,V)$ and $(S',P',V')$ is an extension of $(S,P,V)$ (i.e. $S'\supseteq S$, $P'\sqsupseteq P$, $V'\supseteq V$) and $S'$ is still correct then $(S',P',V')$ has the same $H$-expression and $H$-value as $(S,P,V)$.  Also, note that if the $H$-step applies to $(S,P,V)$ then $(e(S,P,V),?)\in S$.

When $r(S,P,V)\neq\Omega$, $e$ is active at $(S,P,V)$ iff $e=e(S,P,V)$.  When $r(S,P,V)=\Omega$, however, other rank $\Omega$ expressions can also be active: when $e(S,P,V)$ is a formula, all of $S^-_{=\Omega}$ is active, and when $e(S,P,V)$ is a term, many elements from $S^+_{=\Omega}$ may be active.

\begin{definition}
  A \emph{sequent} $\Theta=(S,P,V,F)$ is a historical sequence $(S,P,V)$ together with a function
\[F:\{e\in\dom(S)\mid S(e)={?}\text{ or }rk(e)=\Omega\}\rightarrow\{t,f\}.\]

When $\Theta=(S,P,V,F)$, we write $\Theta_S$ for $S$, $\Theta_P$ for $P$, and so on.  We write $\dom(\Theta)$ for $\dom(S)$.  We often write $(e,u)\in\Theta$ as an abbreviation for $(e,u)\in\Theta_S$.  We write $\Theta t=\{(e,u)\in\Theta_S\mid \Theta_F(e)=t\}$ and $\Theta f=\{(e,u)\in\Theta_S\mid \Theta_F(e)=f\}$.

If $\bowtie\in\{<,\leq,=,\geq,>\}$ then we say $\Theta\bowtie r$ iff for every $e\in\dom(\Theta)$, $rk(e)\bowtie r$.  We write $\Theta_{\bowtie r}$ for the sequent with
\begin{itemize}
\item $(\Theta_{\bowtie r})_S=(\Theta_S)_{\bowtie r}$,
\item $(\Theta_{\bowtie r})_P=\left\{\begin{array}{ll}
\Theta_P&\text{if }\Omega\bowtie r\\
\emptyset&\text{otherwise}
\end{array}\right.$,
\item $(\Theta_{\bowtie r})_V=\left\{\begin{array}{ll}
\Theta_V&\text{if }\Omega\bowtie r\\
\emptyset&\text{otherwise}
\end{array}\right.$,
\item $(\Theta_{\bowtie r})_F=\Theta_F\upharpoonright\dom(\Theta_{\bowtie r})$.
\end{itemize}
\end{definition}

We extend $H(S,P,V)$ to $H(S,P,V,F)$; the only new complication is defining how $F$ should be defined on new expressions appearing in $H(S,P,V)$ but not $S$---$F$ should be equal to $t$ on all such elements.
\begin{definition}
  We say the \emph{$H$-step applies} to $\Theta=(S,P,V,F)$ if the $H$-step applies to $(S,P,V)$.  We write $r(\Theta)$ for $r(S,P,V)$, and similarly for $e(\Theta)$ and $v(\Theta)$.  We define a sequent $H(\Theta)$ as follows:
\begin{itemize}
\item $H(\Theta)_S=H(S,P,V)_S$,
\item $H(\Theta)_P=H(S,P,V)_V$,
\item $H(\Theta)_V=H(S,P,V)_P$,
\item $H(\Theta)_F\supseteq\Theta_F\upharpoonright\dom(H(\Theta)_S)$, and if $e\in\dom(H(\Theta)_S)\setminus \dom(S)$ and either $H(\Theta)_S(e)={?}$ or $rk(e)=\Omega$ then $H(\Theta)_F(e)=t$.
\end{itemize}
$e$ is active at $\Theta$ if $e$ is active at $(S,P,V)$.
\end{definition}

It will be convenient to indicate adding a single element to a sequent by juxtaposition.  We need several cases to cover the different amounts of information needed depending on the form of the expression we are adding.
\begin{definition}
When $e$ is a canonical term, $rk(e)\neq\Omega$, $u\neq{?}$, and $e\not\in\dom(\Theta)$, we write $(e,u),\Theta$ for the sequent $\Theta'$ such that:
\begin{itemize}
\item $\Theta'_S=\Theta_S\cup\{(e,u)\}$,
\item $\Theta'_P=\Theta_P, \Theta'_V=\Theta_V, \Theta'_F=\Theta_F$.
\end{itemize}

When $e$ is a canonical expression, $i\in\{t,f\}$, and $e\not\in\dom(\Theta)$, we write $(e,?,i),\Theta$ for the sequent $\Theta'$ such that:
\begin{itemize}
\item $\Theta'_S=\Theta_S\cup\{(e,?)\}$,
\item $\Theta'_P=\Theta_P, \Theta'_V=\Theta_V$,
\item $\Theta'_F=\Theta_F\cup\{(e,i)\}$.
\end{itemize}

When $e$ is a canonical term with $rk(e)=\Omega$, $i\in\{t,f\}$, $u\neq{?}$, and $e\not\in\dom(\Theta)$, we write $(e,u,i),\Theta$ for the sequent $\Theta'$ such that:
\begin{itemize}
\item $\Theta'_S=\Theta_S\cup\{(e,u)\}$,
\item $\Theta'_P=\Theta_P, \Theta'_V=\Theta_V$,
\item $\Theta'_F=\Theta_F\cup\{(e,i)\}$.
\end{itemize}

If $e$ is an expression of the form $n\in I$ with $e\not\in\dom(\Theta)$, we say $P_e$ is \emph{proper} for $e,\Theta$ if:
\begin{itemize}
\item $P_e$ is a finite sequence of distinct formulas of the form $m\in I$,
\item if $m\in I$ is in $P_e$ and $(c_m,u)\in\Theta$ then $u={?}$,
\item $\Theta_P{}^\frown P_e$ is admissible,
\item the last element of $P_e$ is $e$, and
\item $\Theta_P$ and $P_e$ are disjoint.
\end{itemize}

When $e$ is an expression of the form $n\in I$, $i\in\{t,f\}$, $P_e$ is proper for $e,\Theta$, and $\dom(V_e)$ is exactly those $m\in I$ appearing in $P_e$, we write $(e,\top,i,P_e,V_e)$ for the sequence $\Theta'$ such that:
\begin{itemize}
\item $\Theta'_S=\Theta_S\cup\{(m\in I,\top)\mid m\in I\text{ appears in }P_e\}$,
\item $\Theta'_P=\Theta_P{}^\frown P_e$,
\item $\Theta'_V=\Theta_V\cup V_e$,
\item $\Theta'_F=\Theta_F\cup\{(m\in I,i)\mid m\in I\text{ appears in }P_e\}$.
\end{itemize}
\end{definition}
When we write $(e,u),\Theta$ (or any of its variants), we always assume that $e\not\in\dom(\Theta)$.

We now define a deduction system for our sequents.  The axioms and inference rules are as follows.

\begin{itemize}
\item $AxF$: $\Theta$ is an instance of $AxF$ if $\Theta_S$ is computationally inconsistent,
\item $AxS$: $\Theta$ is an instance of $AxS$ if $\Theta_S$ is solving,
\item $AxH_{e,v}$: $\Theta$ is an instance of $AxH_{e,v}$ if the $H$-step applies to $\Theta$ with $e=e(\Theta)$, and $v=v(\Theta)$, and there is an active $e'\in\dom(\Theta)$ with $\Theta_F(e')=f$,
\item $Cut_e$: When $e$ is a canonical term with $rk(e)\neq\Omega$,

\begin{prooftree*}
    \Hypo{(e,?,f),\Theta}
    \Hypo{\ldots\hspace{-.7em}}
    \Hypo{(e,u),\Theta}
    \Hypo{\hspace{-.7em}\ldots}
    \Hypo{(u\in\mathbb{N})}
    \Infer5{\Theta}
  \end{prooftree*}

\item $CutFr_e$: When $e$ is a canonical term with $rk(e)\neq\Omega$,

\begin{prooftree*}
    \Hypo{(e,?,t),\Theta}
    \Hypo{\ldots\hspace{-.7em}}
    \Hypo{(e,u),\Theta}
    \Hypo{\hspace{-.7em}\ldots}
    \Hypo{(u\in\mathbb{N})}
    \Infer5{\Theta}
  \end{prooftree*}

\item $Cut^{\Omega,\mathrm{term}}_e$: When $e$ is a canonical term with $rk(e)=\Omega$,

\begin{prooftree*}
    \Hypo{(e,?,f),\Theta}
    \Hypo{\ldots\hspace{-.7em}}
    \Hypo{(e,u,f),\Theta}
    \Hypo{\hspace{-.7em}\ldots}
    \Hypo{(u\in\mathbb{N})}
    \Infer5{\Theta}
  \end{prooftree*}

\item $CutFr^{\Omega,\mathrm{term}}_e$: When $e$ is a canonical term with $rk(e)=\Omega$,

\begin{prooftree*}
    \Hypo{(e,?,t),\Theta}
    \Hypo{\ldots\hspace{-.7em}}
    \Hypo{(e,u,f),\Theta}
    \Hypo{\hspace{-.7em}\ldots}
    \Hypo{(u\in\mathbb{N})}
    \Infer5{\Theta}
  \end{prooftree*}

\item $Cut^{\Omega,\mathrm{form}}_e$: When $e$ is a positive rank $\Omega$ formula,

  \begin{prooftree*}
    \Hypo{(e,?,f),\Theta}
    \Hypo{\ldots\hspace{-.7em}}
    \Hypo{(e,\top,f,P_e,V_e),\Theta}
    \Hypo{\hspace{-.7em}\ldots}
    \Hypo{(P_e\text{ is proper for }e,\Theta)}
    \Infer5{\Theta}
  \end{prooftree*}





\item $Fr_e$:

  \begin{prooftree*}
    \Hypo{(e,?,t),\Theta}
    \Infer1{\Theta}
  \end{prooftree*}

\item $H_{e,v}$: 

  \begin{prooftree*}
    \Hypo{H((e,?,t),\Theta)}
    \Infer1{(e,?,t),\Theta}
  \end{prooftree*}

where the $H$-step applies to $(e,?,t),\Theta$ with $e=e(\Theta)$, and $v=v(\Theta)$, and for every $e'$ active in $(e,?,t),\Theta$, $((e,?,t),\Theta)_F(e')=t$.  If $(e',u)\in H((e,?,t),\Theta)\setminus ((e,?,t),\Theta)$ then $H((e,?,t),\Theta)_F(e)=t$.

\end{itemize}

The various $Cut$ and $CutFr$ rules share a common form, so we could combine all three rules into a single rule stated slightly more abstractly; however, for clarity, we have stated the three cases separately instead.


A deduction is a well-founded tree constructed according to the inference rules above.
\begin{definition}
If $\mathcal{S}$ is a set of sequents, a \emph{deduction} from $\mathcal{S}$ is defined inductively:
\begin{itemize}
\item If $\Theta\in\mathcal{S}$ then $\Theta$, by itself, is a deduction of $\Theta$ from $\mathcal{S}$,
\item If $\Theta$ is an axiom then $\Theta$, by itself, is a deduction of $\Theta$ from $\mathcal{S}$,
\item If $\mathrm{I}$ is an inference rule 

  \begin{prooftree*}
    \Hypo{\cdots}
    \Hypo{\Theta_i}
    \Hypo{\cdots}
    \Hypo{(i\in\mathcal{I})}
    \Infer4{\Theta}    
  \end{prooftree*}

and for each $i\in\mathcal{I}$, $d_i$ is a deduction of $\Theta_i$ from $\mathcal{S}$ then $\mathrm{I}(\{d_i\}_{i\in\mathcal{I}})$ is a deduction of $\Theta$ from $\mathcal{S}$.
\end{itemize}

A \emph{derivation} is a deduction from $\emptyset$.
\end{definition}

We generally view deductions starting at the root and moving towards the leaf: our goal is to obtain a derivation of $\emptyset$ consisting of only $Fr$ and $H$ inferences, since, as we will show below, such a derivation will prove that the $H$-process terminates.  We will start with a derivation of $\emptyset$ consisting entirely of $Cut$ (or its variants, $Cut^{\Omega,\mathrm{term}}, Cut^{\Omega,\mathrm{form}}$), so as we move away from the root we consider all possible values that an expression $e$ could be assigned.

These rules enforce that when $\Theta_F(e)=f$, $e$ is ``fixed''---as we move further away from the root, a value $(e,?)$ with $\Theta_F(e)={?}$ will never change.  On the other hand when $\Theta_F(e)=t$, we could encounter an $H$ inference which removes $(e,?)$ and replaces it with $(e,v)$.

  In any rule $X_e$, we refer to $e$ as the \emph{main expression} of the rule.  If $d$ is a deduction ending in a $Cut_e$ or $CutFr_e$ rule (or an $\cdot^\Omega$ variant), we generally refer to the immediate subdeduction of $(e,u),\Theta$ as $d_u$.

\section{A Starting Derivation}

The construction of the original derivation in this section is the same as in \cite{MintsTupailoBuchholz1996} for $\epsilon$-terms $e$, the step described in Lemma \ref{thm:computeterm}.  In Lemma \ref{thm:computeformula} we use $Cut^{\Omega,\mathrm{form}}_e$ inferences to decide formulas $n\in I$.

Essentially, we attempt to evaluate the critical formulas; if our sequent is not already an axiom then it must be computationally consistent, non-solving, and the $H$-rule does not apply, so we will find canonical expressions which have not been assigned values.  We apply a cut over some canonical expression appearing in our critical formulas, and repeat the process for every premise of the cut.  In order to show that the process halts, we always choose canonical subexpressions of formulas having the maximum possible rank.  A small complication is presented by the fact that the evaluation of formulas in this way could lead us to something ill-founded if we insist on evaluating everything (for instance, if the formula $A(n,I)$ includes $n+1\in I$ as a subformula, we cannot expect to have a well-founded tree consisting of $Cut$-inferences where all $AxH$ inferences are actually computing).

This construction depends on the fixed set of critical formulas $\mathrm{Cr}=\{Cr_0,\ldots,Cr_N\}$.

\begin{definition}
  For any sequent $\Theta$ and any set of formulas $C$, let 
\[F(\Theta,C)=C\cup\{F(e,u)\mid (e,u)\in\Theta_S,\, e\text{ is a term}\}.\]

Define:
\begin{itemize}
\item $\rho(\Theta,C)=\max\{rk_s(|\phi|_{\Theta_S})\mid \phi\in F(\Theta,C)\}$,
\item $d(\phi)$ is the total number of formulas $t\in I$ and Skolem terms appearing in $\phi$,
\item $d_r(\phi)=\left\{\begin{array}{ll}
      d(\phi)&\text{if }rk_s(\phi)=r\\
      0&\text{otherwise}\end{array}\right.$
\item $\nu(\Theta,C)=\omega\cdot \rho(\Theta)+\#_{\phi\in F(\Theta,C)}d_{\rho(\Theta,C)}(|\phi|_{\Theta_S})$.
\end{itemize}
\end{definition}
Note that if $S\subseteq S'$ then $d(|\phi|_{S'})\leq d(|\phi|_S)$, and if $e$ is a canonical expression appearing in $|\phi|_S$ and $e\in\dom(S')$ then $d(|\phi|_{S'})<d(|\phi|_S)$.  Therefore the quantity $\nu(\Theta)$ describes the highest rank appearing in $F(\Theta)$ (the $\omega\cdot\rho(\Theta)$ term) together with the number of expressions we might have to add to $\Theta$ before we can reduce that rank.

\begin{lemma}
  If $\Theta$ is a sequent, $\phi\in F(\Theta,C)$ with $rk_s(|\phi|_{\Theta_S})=\rho(\Theta,C)$, and $e$ is a canonical term appearing in $|\phi|_{\Theta_S}$ with $rk(e)\neq\Omega$, then there is a deduction of $\Theta$ from sequents $\Theta'$ with $\nu(\Theta',C)<\nu(\Theta,C)$ consisting entirely of $Cut$-inferences of rank $\leq\Omega+\rho(\Theta,C)$.
  \label{thm:computeterm}
\end{lemma}
\begin{proof}
  We deduce $\Theta$ from sequents $(e,u),\Theta$ using a $Cut_e$-inference.  It suffices to show that $\nu(((e,u),\Theta),C)<\nu(\Theta,C)$.

  Observe that $F(((e,u),\Theta),C)=F(\Theta,C)\cup\{F(e,u)\}$.  If $\rho(((e,u),\Theta),C)<\rho(\Theta,C)$ then certainly $\nu(((e,u),\Theta),C)<\nu(\Theta,C)$.  Suppose $\rho(((e,u),\Theta),C)<\rho(\Theta,C)$; then since $rk_s(F(e,u))<rk_s(e)\leq\rho(\Theta,C)$, $F(e,u)$ does not contribute to $\nu((e,u),\Theta)$.  For each $\psi\in F(\Theta,C)$, we have $rk_s(|\psi|_{((e,u),\Theta)_S})\leq rk_s(|\psi|_{\Theta_S})$, and if $rk_s(|\psi|_{((e,u),\Theta)_S})= rk_s(|\psi|_{\Theta_S})$ then for every subterm in $|\psi|_{((e,u),\Theta)_S}$ there is a corresponding term with rank at least as high in.  Moreover, $d(|\phi|_{((e,u),\Theta)_S}|)<d(|\phi|_{\Theta_S})$, so $\nu(((e,u),\Theta),C)<\nu(\Theta,C)$.
\end{proof}

\begin{lemma}
  If $\Theta$ is a sequent, $\phi\in F(\Theta,C)$, and $e$ is a canonical formula appearing in $|\phi|_{\Theta_S}$ then there is a deduction of $\Theta$ from sequents $\Theta'$ with $\nu(\Theta',C)\leq\nu(\Theta,C)$ consisting entirely of $Cut$-inferences of rank $\leq\Omega$.  Furthermore, if $rk_s(|\phi|_{\Theta_S})=\rho(\Theta,C)$ then the deduction of $\Theta$ is from sequents $\Theta'$ with $\nu(\Theta',C)<\nu(\Theta,C)$.
  \label{thm:computeformula}
\end{lemma}
\begin{proof}
We deduce $\Theta$ from sequents $(e,?,f),\Theta$ and $(e,\top,f,P_e,V_e),\Theta$ using a $Cut_e^{\Omega,\mathrm{form}}$-inference.  The claim follows as in the previous lemma.
\end{proof}

\begin{lemma}
  If $\Theta$ is a sequent, $\phi\in F(\Theta,C)$ with $rk_s(|\phi|_{\Theta_S})=\rho(\Theta)$, and $e$ is a canonical term appearing in $|\phi|_{\Theta_S}$ with $rk(e)=\Omega$, then there is a deduction of $\Theta$ from sequents $\Theta'$ with $\nu(\Theta',C)<\nu(\Theta,C)$ consisting entirely of $Cut$-inferences of rank $\leq\rho(\Theta,C)$.
  \label{thm:computeterm_Omega}
\end{lemma}
\begin{proof}
  We deduce $\Theta$ from sequents $(e,u),\Theta$ using a $Cut_e$-inference.  For each $(e,u),\Theta$, we apply the previous lemma repeatedly (at most once for each $n\in I$ appearing in $F(e,u)$), to obtain leaves $\Theta'$ where $rk_s(|F(e,u)|_{\Theta'_S})=0$.  Then the claim follows as in Lemma \ref{thm:computeterm}.
\end{proof}

\begin{lemma}
  For any $\Theta$ and $C$, there is a deduction of consisting of $Cut$-inferences of rank $\leq \Omega+\rho(\Theta,C)$ from correct sequents $\Theta'$ such that for each $\phi\in F(\Theta,C)$, $\Theta'_S$ decides $\phi$.
\end{lemma}
\begin{proof}
    By induction on $\nu(\Theta,C)$ and repeated applications of Lemma \ref{thm:computeterm}, Lemma \ref{thm:computeformula}, and Lemma \ref{thm:computeterm_Omega}, we obtain a deduction of $\Theta$ from sequents $\Theta'$ where, for every $\phi\in F(\Theta')$ with $rk_s(|\phi|_{\Theta'_S})=\rho(\Theta')$, there are no canonical expressions in $|\phi|_{\Theta'_S}$.  Such a sequent must have $|\phi|_{\Theta'_S}$ a formula in the language of Peano arithmetic for each $\phi\in F(\Theta,C)$, so $\Theta'_S$ decides each such $\phi$.  If $\Theta'$ is correct, we are done.

If $\Theta'$ is not correct, it must be because there is some $(e,u)\in\Theta'$ with $\overline{\Theta'_S}\vDash\neg F(e,u)$; if $F(e,u)\in F(\Theta,C)$ then $\Theta'_S$ decides $F(e,u)$, so $\Theta'_S\vDash \neg F(e,u)$, so $\Theta'$ is an instance of $AxF$.  Otherwise $e$ must be a formula $n\in I$ and $u=\top$.  If $c_n\in\dom(\Theta')$, if $\Theta'_S(c_n)\neq{?}$ then again $\Theta'$ is an instance of $AxF$.  We cannot have $\Theta'_S(c_n)={?}$ or $c_n\not\in\dom(\Theta')$: since $\Theta'_S$ decides $B(0,n,I)$ then properness of $\Theta'_P$ implies admissibility of $\Theta'_P$, which (together with the monotonicity of $B$) implies that $|B(0,n,I)|_{\Theta'_S}$ must be true, so $\Theta'_S\vDash F(e,u)$, contradicting the assumption that $\overline{\Theta'_S}\vDash\neg F(e,u)$.


\end{proof}

\begin{lemma}
  For any $\Theta$, there is a deduction of consisting of $Cut$-inferences of rank $\leq \Omega+\rho(\Theta,C)$ from correct sequents $\Theta'$ such that:
\begin{itemize}
\item for each $\phi\in F(\Theta,C)$, $\Theta'_S$ decides $\phi$.
\item if $Cr_I$ is an \textbf{Induction} axiom $\phi(0,\vec s)\wedge\neg\phi(t,\vec s)\rightarrow\exists x(\phi(x,\vec s)\wedge\neg\phi(Sx,\vec s))$ such that $\Theta'\vDash\neg Cr_I$ then for all $m\leq |t|_{\Theta'_S}+1$, $\Theta'_S$ decides $\phi(m,\vec s)$.
\end{itemize}
\end{lemma}
\begin{proof}
  By the previous lemma applied to $\Theta,C\cup \mathrm{Cr}$, there is a deduction of $\Theta$ from sequents satisfying the first clause.  So assume $\Theta$ already satisfies the first clause.  Call an \textbf{Induction} axiom $Cr_I$ \emph{unfinished} at $\Theta$ if the second clause does not hold for $Cr_I$.  We proceed by induction on the number of unfinished \textbf{Induction} axioms in $\Theta$; if there are none, we are finished.

  If $\phi(0,\vec s)\wedge\neg\phi(t,\vec s)\rightarrow\exists x(\phi(x,\vec s)\wedge\neg\phi(Sx,\vec s))$ is an unfinished axiom, the first assumption implies that $|t|_{\Theta_S}$ is a numeral $n$.  We apply the previous lemma to $\Theta, \{\phi(m,\vec s)\mid m\leq n+1\}$ to obtain a deduction of $\Theta$ from sequents satisfying the first clause.  Any axiom finished in $\Theta$ will be satisfied in all sequents $\Theta'$, so the claim follows by the inductive hypothesis.
\end{proof}

\begin{lemma}
  For any $\Theta$, there is a deduction of consisting of $Cut$-inferences of rank $\leq \Omega+\rho(\Theta,C)$ from sequents $\Theta'$ such that:
\begin{itemize}
\item for each $\phi\in F(\Theta,C)$, $\Theta'_S$ decides $\phi$.
\item if $Cr_I$ is an \textbf{Induction} axiom $\phi(0,\vec s)\wedge\neg\phi(t,\vec s)\rightarrow\exists x(\phi(x,\vec s)\wedge\neg\phi(Sx,\vec s))$ such that $\Theta'\vDash\neg Cr_I$ then for all $m\leq |t|_{\Theta'_S}+1$, $\Theta'_S$ decides $\phi(m,\vec s)$,
\item if $Cr_I$ is a \textbf{Closure} axiom $\forall x(A(x,\phi(x))\rightarrow \phi(x))\rightarrow\forall x(x\in I\rightarrow\phi(x))$ such that $\Theta'\vDash\neg Cr_I$ then, taking $n\in I$ to be the first formula in $P$ such that $\Theta'\vDash\neg\phi(n)$, for all $m\in I$ appearing before $n\in I$ in $P$ (including $n$), $\Theta'$ decides $\phi(m)$ and $A(m,\phi(m))$.
\end{itemize}
\end{lemma}
\begin{proof}
  By the previous lemma applied to $\Theta,C\cup\mathrm{Cr}$, there is a deduction of $\Theta$ from sequents satisfying the first two clauses.  Again we call $Cr_I$ \emph{unfinished} if it is a \textbf{Closure} axiom violating the third clause.  We proceed by induction on the number of unfinished axioms; if there are none, we are finished.

  Suppose $Cr_I=\forall x(A(x,\phi(x))\rightarrow \phi(x))\rightarrow\forall x(x\in I\rightarrow\phi(x))$ is an unfinished axiom.  Since $\Theta\vDash\neg Cr_I$, we must have $\Theta\vDash \exists x(x\in I\wedge\neg\phi(x))$, so there is some $n\in I$ in $P$ so that $\Theta\vDash\neg\phi(n)$.  We apply the previous lemma to $\Theta,\{\phi(m),A(m,\phi(m))\mid m\in I\text{ appears before }n\in I\text{ in }P\}\cup\{\phi(n),A(n,\phi(n))\}$.  Any axiom finished in $\Theta$ will be finished in all sequents $\Theta'$, so the claim follows by the inductive hypothesis.
\end{proof}

\begin{theorem}\label{thm:initial_deriv}
  There is a derivation of $\emptyset$ consisting of $Cut$-inferences of rank $\leq\Omega+\rho(\emptyset,\mathrm{Cr})$.
\end{theorem}
\begin{proof}
  We apply the previous lemma to $\emptyset,\mathrm{Cr}$.  We claim that all sequents $\Theta'$ satisfying the three clauses in that lemma are axioms.  Suppose $\Theta'$ is not an $AxF$ or $AxS$ axiom.  The only way the $H$-rule fails to apply to $\Theta'$ is if there is a $(c_n,?)\in\Theta'_S$ such that $\Theta'$ fails to decide $n\in I$; there are finitely many such $n\in I$, so by finitely many applications of Lemma \ref{thm:computeformula}, we may derive $\Theta'$ from instances of the $AxH$ axiom.
\end{proof}

\section{Predicative Cut-Elimination}
\subsection{Outline of Predicative Cut-Elimination}

\begin{definition}
  Let $d$ be a deduction.  If $X\in\{Cut,CutFr,Fr,H,HS\}$ and $\bowtie\in\{\leq,<,=,>,\geq\}$ then we say $X(d)\bowtie r$ if every application of a rule $X$, $X^{\Omega,\mathrm{term}}$, or $X^{\Omega,\mathrm{form}}$ in $d$ has main expression with rank $\bowtie r$.  

  We say $d$ is an \emph{$r$-deduction} if $Cut(d)<r$, $CutFr(d)<0$, $Fr(d)\geq r$, and $H(d)\geq r$.

We say $d$ is an \emph{$r^+$-deduction} if $Cut(d)<r$, $CutFr(d)=r$, $Fr(d)>r$, and $H(d)\geq r$.
\end{definition}

The basic operation in the cut-elimination proof is to take an $r+1$-derivation of $\emptyset$ and convert it to an $r$-derivation of $\emptyset$.  As an intermediate step, we produce an $r^+$-derivation, since an $r^+$-derivation may be pruned (Lemma \ref{thm:pred_prune}) to an $r$-derivation.  The conversion of an $r+1$-derivation to an $r^+$-derivation takes place inductively, beginning at the axioms.  The main step occurs at the $Cut_e$ rule with $rk(e)=r$; suppose we have, inductively, converted all branches of some $Cut_e$ rule to $r^+$-derivations:

  \begin{prooftree*}
    \Infer0[$AxH_{e,u}$]{\tikzmark{pt1_s1}(e,?,f),\Sigma}
    \Ellipsis{}{\tikzmark{pt1_e1}(e,?,f),\Theta}
    \Hypo{\ldots}
    \Hypo{\phantom{(e,u),\Theta}\tikzmark{pt1_s2}}
    \Ellipsis{}{(e,u),\Theta\tikzmark{pt1_e2}}
    \Infer3[$Cut_e$]{\Theta\tikzmark{pt1_s3}}
    \Ellipsis{}{\emptyset\tikzmark{pt1_e3}}
  \end{prooftree*}

  \begin{tikzpicture}[remember picture, overlay]
    \draw[decorate, decoration = {brace, amplitude=6pt,mirror}] (pic cs:pt1_s1) -- ($ (pic cs:pt1_e1) + (0,.23) $) node[midway, left=3pt] {$r^+$-derivation $d_?$};
    \draw[decorate, decoration = {brace, amplitude=6pt}] (pic cs:pt1_s2) -- ($ (pic cs:pt1_e2) + (0,.23) $) node[midway, right=3pt] {$r^+$-derivation $d_u$};
    \draw[decorate, decoration = {brace, amplitude=6pt}] (pic cs:pt1_s3) -- ($ (pic cs:pt1_e3) + (0,.23) $) node[midway, right=3pt] {$r+1$-derivation};
  \end{tikzpicture}

We should replace the $Cut_e$ rule with a $CutFr_e$ rule, which is permitted in an $r^+$-derivation.  This means setting $e$ to be temporary everywhere in the branch above $(e,?,t),\Theta$ (that is, if $\Theta'$ is some sequent appearing in the branch above $(e,?,t),\Theta$, we should replace it with $\Theta''$ which is identical except that $\Theta''_F(e)=f$); this is Lemma \ref{thm:f_to_t}.  This invalidates all $AxH_{e,u}$ axioms above $(e,?,f),\Theta$; these axioms must be replaced by $H_{e,u}$ inferences, and we must produce an $r^+$-derivation of $(e,u)$ to place above them:

  \begin{prooftree*}
    \Hypo{?}
    \Infer1{(e,u),\Sigma_{\leq r}}
    \Infer1[$H_{e,u}$]{\tikzmark{pt2_s1}(e,?,t),\Sigma}
    \Ellipsis{}{\tikzmark{pt2_e1}(e,?,t),\Theta}
    \Hypo{\ldots}
    \Hypo{\phantom{(e,u),\Theta}\tikzmark{pt2_s2}}
    \Ellipsis{}{(e,u),\Theta\tikzmark{pt2_e2}}
    \Infer3[$CutFr_e$]{\Theta\tikzmark{pt2_s3}}
    \Ellipsis{}{\emptyset\tikzmark{pt2_e3}}
  \end{prooftree*}

  \begin{tikzpicture}[remember picture, overlay]
    \draw[decorate, decoration = {brace, amplitude=6pt,mirror}] (pic cs:pt2_s1) -- ($ (pic cs:pt2_e1) + (0,.23) $) node[midway, left=3pt] {$r^+$-derivation $d_?$};
    \draw[decorate, decoration = {brace, amplitude=6pt}] (pic cs:pt2_s2) -- ($ (pic cs:pt2_e2) + (0,.23) $) node[midway, right=3pt] {$r^+$-derivation $d_u$};
    \draw[decorate, decoration = {brace, amplitude=6pt}] (pic cs:pt2_s3) -- ($ (pic cs:pt2_e3) + (0,.23) $) node[midway, right=3pt] {$r+1$-derivation};
  \end{tikzpicture}

If $\Theta$ were equal to $\Sigma_{\leq r}$, we could simply place $d_u$ as $?$.  We will have to make some modifications to $d_u$ to make it a deduction of $(e,u),\Sigma_{\leq r}$.  There are two sources of differences between $\Sigma_{\leq r}$ and $\Theta$, corresponding to ways $d_?$ might have modified expressions of rank $\leq r$ and ways $d_?$ might have modified expressions of rank $> r$.

We consider the $\leq r$ case first; the derivation $d_?$ is an $r^+$-derivation, so the only inference rules modifying the expressions of rank $<r$ are $Cut_{e'}$ inferences.  There might be both $CutFr_{e'}$ and $H_{e',v}$ inferences with $rk(e')=r$; importantly, however, if an $H_{e',v}$ inference appears, we have $(e',?,t)\not\in\Theta$: the branch $d_?$ can introduce new rank $e'$ terms with a $CutFr_{e'}$, and then higher-up those terms might be modified with $H_{e',v}$ inferences, but the branch does not remove any expressions of rank $\leq r$ from $\Theta$.  (This property is formalized in Lemma \ref{thm:pred_char}.)  We attempt to copy these over to $d_u$ by replacing each sequent $\Theta'$ in $d_u$ with $\Theta'\cup\Sigma_{\leq r}$; the only obstacle is if there is a $Cut_{e'}$ or $CutFr_{e'}$ inference in $d_u$ with $e'\in \dom(\Sigma_{\leq r}\setminus\Theta)$: if $(e',v)\in\Sigma_{\leq r}$ then

  \begin{prooftree}
\Hypo{}
\Ellipsis{}{(e',?,t),(e,u),\Delta'}
\Hypo{\ldots}
\Hypo{}
\Ellipsis{}{(e',v),(e,u),\Delta''}
\Ellipsis{}{(e',v),(e,u),\Delta'}
\Infer3[$CutFr_{e'}$]{(e,u),\Delta'}
\Ellipsis{}{(e,u),\Delta}
\Ellipsis{}{(e,u),\Theta}
  \end{prooftree}
becomes
  \begin{prooftree}
\Hypo{}
\Ellipsis{}{(e',v),(e,u),\Delta''\cup\Sigma_{\leq r}}
\Ellipsis{}{(e',v),(e,u),\Delta'\cup\Sigma_{\leq r}}
\Ellipsis{}{(e,u),\Delta\cup\Sigma_{\leq r}}
\Ellipsis{}{(e,u),\Theta\cup\Sigma_{\leq r}}
  \end{prooftree}

(Note that, since $(e',v)\in\Sigma_{\leq r}$, we include it in the deduction on the right only for clarity: $(e',v),(e,u),\Delta''\cup\Sigma_{\leq r}=(e,u),\Delta''\cup\Sigma_{\leq r}$.)  This transformation is the content of Lemma \ref{thm:pred_compat}.  This step is why we need to use the $CutFr$ inference as an intermediate step: during the process of converting an $r+1$-derivation to an $r^+$-derivation, we need to keep the side branches around.  After we've converted the whole deduction to an $r^+$-derivation, the side branches are no longer needed, so we prune them to obtain an $r$-derivation.

We turn to the second class of differences between $\Sigma_{\leq r}$ and $\Theta$: $\Theta$ contains terms of rank $>r$ which were deleted in $\Sigma_{\leq r}$.  To restore these, we produce a deduction of $(e,u),\Sigma_{\leq r}$ from $(e,u),\Theta\cup\Sigma_{\leq r}$.  To do this, we consider the path from $\Theta$ down to $\emptyset$.  The only inference rules in this path affecting expressions of rank $>r$ are $H_{e',v}$ and $Fr_{e'}$ rules.  In Lemma \ref{thm:pred_rep} we simply repeat the same sequence of $H_{e',v}$ and $Fr_{e'}$ rules (ignoring all $Cut$ rules in the path, since any expression introduced in a $Cut$ rule in $p$ is still present in $\Sigma_{\leq r}$):

  \begin{prooftree*}
    \Hypo{\tikzmark{pt3_v1}\phantom{(e,u),\Theta\cup\Sigma_{\leq r}}}
    \Ellipsis{}{\tikzmark{pt3_s1}(e,u),\Theta\cup\Sigma_{\leq r}}
    \Ellipsis{}{\tikzmark{pt3_e1}(e,u),\Sigma_{\leq r}}
    \Infer1[$H_{e,u}$]{\tikzmark{pt3_s2}(e,?,t),\Sigma}
    \Ellipsis{}{\tikzmark{pt3_e2}(e,?,t),\Theta}
    \Hypo{\ldots}
    \Hypo{\phantom{(e,u),\Theta}\tikzmark{pt3_s3}}
    \Ellipsis{}{(e,u),\Theta\tikzmark{pt3_e3}}
    \Infer3[$CutFr_e$]{\Theta\tikzmark{pt3_s4}}
    \Ellipsis{}{\emptyset\tikzmark{pt3_e4}}
  \end{prooftree*}

  \begin{tikzpicture}[remember picture, overlay]
    \draw[decorate, decoration = {brace, amplitude=6pt,mirror}] (pic cs:pt3_v1) -- ($ (pic cs:pt3_s1) + (0,.23) $) node[midway, left=3pt] {$d_u\cup\Sigma_{\leq r}$};
    \draw[decorate, decoration = {brace, amplitude=6pt,mirror}] (pic cs:pt3_s1) |- node[pos=.38, left=3pt] {repetition of $p$} ($ (pic cs:pt3_e1) + (0,.23) $) ;
    \draw[decorate, decoration = {brace, amplitude=6pt}] (pic cs:pt3_s3) -- ($ (pic cs:pt3_e3) + (0,.23) $) node[midway, right=3pt] {$d_u$};
    \draw[decorate, decoration = {brace, amplitude=6pt}] (pic cs:pt3_s4) -- ($ (pic cs:pt3_e4) + (0,.23) $) node[midway, right=3pt] {path $p$};
  \end{tikzpicture}

\subsection{Predicative Cut-elimination}

\begin{lemma}\label{thm:pred_prune}
  If $d$ is an $r^+$-deduction of $\Theta$ from $\mathcal{S}$ then there is an $r$-deduction of $\Theta$ from $\mathcal{S}$.
\end{lemma}
\begin{proof}
  Prune every $CutFr$ inference to an $Fr$ inference.
\end{proof}

\begin{lemma}\label{thm:pred_tail_char}
  If $\Theta$ is a sequent in an $r+1$-deduction of $\emptyset$ then $\Theta t>r$ and $\Theta f\leq r$.
\end{lemma}
\begin{proof}
  By induction on $d$ beginning from the root.  This is trivial for $\emptyset$.  If $\Theta$ appears in the premise of a $Cut_e$ or $Cut^\Omega_e$ rule with conclusion $\Theta'$ then $\Theta t=\Theta' t$ and $\Theta f\subseteq \Theta' f\cup\{e\}$, and $rk(e)\leq r$.  If $\Theta$ appears in the premise of a $Fr_e$ rule with conclusion $\Theta'$ then $\Theta t=\Theta' t\cup\{e\}$ and $rk(e)\geq r+1$ and $\Theta f=\Theta' f$.  If $\Theta$ appears in the premise of an $H_{e,v}$ rule with $rk(e)\neq\Omega$ with conclusion $\Theta'$ then $\Theta t\subseteq \Theta' t$ and $\Theta f=\Theta' f$.  If $\Theta$ appears in the premise of an $H_{e,v}$ rule with $rk(e)=\Omega$, still $\Theta f=\Theta' f$, and if $e\in \Theta t\setminus\Theta' t$ then $rk(e)=\Omega>r$.
\end{proof}

\begin{lemma}\label{thm:pred_char}
  If $\Sigma$ is a sequent in an $r^+$-deduction of $\Theta$ with $r\neq\Omega$ then:
  \begin{itemize}
  \item $\Theta_{\leq r}\setminus\Theta t\subseteq\Sigma$,
  \item $(\Sigma f)_{\geq r}\subseteq\Theta$,
  \item If $\Theta t\geq r$ then $\Sigma t\geq r$,
  \item if $r>\Omega$ then $\Theta_P$ is an initial segment of $\Sigma_P$.
  \end{itemize}
\end{lemma}
\begin{proof}
  Consider the path from $\Theta$ up to $\Sigma$.  The only inference rule which could cause an element of $\Theta_{\leq r}$ to be removed is $H$, which can only happen to elements in $\Theta t$.  The only way to add a fixed element is with a $Cut$ inference, but all $Cut$ inferences have rank $<r$.  The only way to add a temporary element is with a $CutFr$ or $Fr$ rule, which have rank $=r$ and $>r$ respectively, so if $\Theta t\geq r$ then $\Sigma t\geq r$.  The only way to modify the $P$ component in an $r^+$-deduction with $r>\Omega$ is with a $Cut^{\Omega,\mathrm{form}}$ rule.
\end{proof}

\begin{definition}
  Let $\Theta$ and $\Sigma$ be two sequents.  Then $\Theta$ and $\Sigma$ are \emph{compatible} if:
  \begin{itemize}
  \item $\Theta_S\cup\Sigma_S$ is a function,
  \item $\Theta_F\cup\Sigma_F$ is a function,
  \item Either $\Theta_P$ extends $\Sigma_P$ or $\Sigma_P$ extends $\Theta_P$,
  \item For each $n\in I$ in both shorter of $\Theta_P$ and $\Sigma_P$, either $\Theta_V(n\in I)\subseteq\Sigma_V(n\in I)$ or $\Sigma_V(n\in I)\subseteq\Theta_V(n\in I)$.
  \end{itemize}
If $\Theta$ and $\Sigma$ are compatible, we write $\Theta\cup\Sigma$ for the componentwise union---$(\Theta\cup\Sigma)_S=\Theta_S\cup\Sigma_S$, $(\Theta\cup\Sigma)_F=\Theta_F\cup\Sigma_F$, $(\Theta\cup\Sigma)_P$ is the longer of $\Theta_P$ and $\Sigma_P$, and $(\Theta\cup\Sigma)_V(n\in I)=\Theta_V(n\in I)\cup\Sigma_V(n\in I)$ (where $\Theta_V(n\in I)=\emptyset$ if $n\in I\not\in\dom(\Theta_V)$, and similarly for $\Sigma_V$).
\end{definition}

\begin{lemma}
  If $\Theta$ and $\Sigma$ are compatible then $\Theta\cup\Sigma$ is a sequent.
\end{lemma}

\begin{lemma}\label{thm:pred_compat}
  Suppose $d$ is an $r^+$-derivation of $\Theta$ with $r\neq\Omega$.  Let $\Sigma\leq r$ be a correct sequent such that $\Theta$ and $\Sigma$ are compatible, $(\Sigma f)_{\geq r}\subseteq \Theta$, and $\Sigma t\geq r$.  Then there is an $r^+$-derivation $d\cup\Sigma$ of $\Theta\cup\Sigma$.
\end{lemma}
\begin{proof}
  By induction on $d$.  Consider the final inference of $d$.

Suppose the final inference is $Cut_e$ or $Cut_e^{\Omega,\mathrm{term}}$, so $rk(e)<r$.  If $e\not\in\dom(\Sigma)$ then for each $u$, we apply the inductive hypothesis to obtain $d_u\cup\Sigma$ and then apply a $Cut_e$.  We cannot have $e\in\Sigma t$ since $rk(e)<r$ and $\Sigma t\geq r$.  If $e\in\dom(\Sigma)$ then $d_{\Sigma_S(e)}\cup\Sigma$ is the desired deduction.

The $Cut_e^{\Omega,\mathrm{form}}$ case is similar.  If $e\in\dom(\Sigma)$ and $\Sigma_S(e)={?}$ then we take $d_?\cup\Sigma$ as the desired deduction.  If $\Sigma_S(e)=\top$ then we let $P'$ be the initial segment of $\Sigma_P$ ending in $e$, and let $P_e=P'\setminus \Theta_P$; we then take $d_{P_e,V\upharpoonright P_e}\cup\Sigma$ as the desired deduction.  If $e\not\in\dom(\Sigma)$ and $\Sigma_P\sqsubseteq\Theta_P$, we proceed as for a $Cut_e$, applying a $Cut^{\Omega,\mathrm{form}}_e$ inference to the deductions $d_{P_e,V_e}\cup\Sigma$.  If $\Sigma_P$ properly extends $\Theta_P$, we have to trim the branches; let $P=\Sigma_P\setminus\Theta_P$, and for each $P_e$ proper for $e,\Theta\cup\Sigma$, let $P'_e$ be $P{}^\frown P_e$ and let $V'_e=V_e\cup(\Sigma_V\upharpoonright P)$.  We apply a $Cut^{\Omega,\mathrm{form}}_e$ inference to the deductions $d_{P'_e,V'_e}\cup\Sigma$.

Suppose the final inference is $CutFr_e$, so $rk(e)=r$.  If $e\not\in\dom(\Sigma)$ then again, for each $u$, we apply the inductive hypothesis to obtain $d_u\cup\Sigma$ and then apply a $CutFr$.  We cannot have $e\in \Sigma f$ since $rk(e)\geq r$ and $(\Sigma f)_{\geq r}\subseteq\Theta$.  If $e\in\dom(\Sigma)$ then $d_{\Sigma_S(e)}\cup\Sigma$ is the desired deduction.

Suppose the final inference is $Fr_e$, so $rk(e)>r$.  Then $e\not\in\dom(\Sigma)$, so we apply $Fr_e$ to $d'\cup\Sigma$.

Suppose the final inference is $H_{e,v}$, so $rk(e)\geq r$, $\Theta=(e,?,t),\Upsilon$, and the premise is $\Theta'=(e,u),\Upsilon_{\leq rk(e)}$.  Since $\Sigma\leq r\leq rk(e)$, removing $e$ from $\Sigma$ cannot change the correctness of $\Sigma$, so let $\Sigma'=\Sigma-\{e\}$, so $(\Sigma' f)_{\geq r}\subseteq\Theta'$.  Therefore $\Sigma'$ and $\Theta'$ are compatible, so we may apply the inductive hypothesis to obtain a derivation $d'\ast\Sigma'$ of $\Theta'\ast\Sigma'$, and $d\ast\Sigma$ is obtained by applying $H_{e,v}$ to $d'\ast\Sigma'$.

Note that any axiom remains an axiom after a union with $\Sigma$ (though it is possible for one axiom to change to a different axiom).
\end{proof}

\begin{definition}
  An $r$-path is a sequence $\Theta_0,\ldots,\Theta_n$ with $\Theta_0=\emptyset$ appearing as a path in some $r$-deduction.
\end{definition}

\begin{lemma}[Repetition]\label{thm:pred_rep}
  If $(\Theta_0,\ldots,\Theta_n)$ is an $r+1$-path with $\Theta=\Theta_n$, $\Sigma\leq r$ is a correct sequent, $\Theta_{\leq r}\subseteq \Sigma$, and if $r\geq\Omega$ then $\Theta_P$ is an initial segment of $\Sigma_P$ and for each $n\in I$ in $\dom(\Theta_V)$, $\Theta_V(n\in I)\subseteq\Sigma_V(n\in I)$, then $\Theta$ and $\Sigma$ are compatible and there is a deduction of $\Sigma$ from $\Theta\cup\Sigma$ consisting only of $Fr$ and $H$ inferences of rank $>r$.
\end{lemma}
\begin{proof}
  By induction on $n$.  This is trivial when $n=0$.  Suppose $n>0$; let $\Theta^\prime=\Theta_{n-1}$.  The inference rule between $\Theta_n$ and $\Theta_{n-1}$ is one which can occur in an $r+1$-deduction, so $\Theta^\prime_{\leq r}\subseteq\Theta_{\leq r}$.  Therefore the inductive hypothesis applies to $\Theta^\prime$, and there is a deduction of $\Sigma$ from $\Theta^\prime\cup\Sigma$.  Consider the inference between $\Theta$ and $\Theta^\prime$:
  \begin{itemize}
  \item $Cut_e$, $Cut^{\Omega}_e$: Then $(e,u)\in\Theta_S$ for some $u$, and since $rk(e)\leq r$, $(e,u)\in\Sigma_S$, so $\Theta^\prime\cup\Sigma=\Theta\cup\Sigma$.
  \item $Fr_e$: Since $rk(e)>r$, $e\not\in\dom(\Sigma_S)$, so an $Fr_e$ axiom derives $\Theta^\prime\cup\Sigma$ from $\Theta\cup\Sigma$.
  \item $H_{e,u}$: $\Theta=(e,u),\Upsilon_{\leq rk(e)}$ where $\Theta^\prime=(e,?,t),\Upsilon$.  Since $rk(e)\geq r+1$, $(e,?)\not\in\Sigma_S$, so $H_{e,u}$ gives a deduction of $\Theta^\prime\cup\Sigma$ from $\Theta\cup\Sigma$.
  \end{itemize}
\end{proof}

\begin{lemma}\label{thm:f_to_t}
  Let $d$ be a deduction of $(e,u,f),\Theta$.  Then there is a derivation $d'$ of $(e,u,t),\Theta$ from sequents $\Sigma$ where the $H$-rule applies and $e$ is active, such that every inference rule in $d'$ appears in $d$.
\end{lemma}
\begin{proof}
  By induction on $d$.  Consider the last inference; if the last inference is anything other than an $AxH_{e,v}$, we apply the inductive hypothesis to all premises and then the same inference rule.

  If $(e,u,f),\Theta$ is an instance of $AxH_{e,v}$ and $e$ is not active then $(e,u,t),\Theta$ is also an instance of $AxH_{e,v}$.  If $e$ is active then we may leave $(e,u,t),\Theta$ as a leaf.
\end{proof}

\begin{theorem}
  Let $d$ be a derivation of $\Theta$ ending in a $Cut_e$ with $rk(e)=r\neq\Omega$ such that the immediate subderivations are $r^+$-derivations and there is an $r+1$-path for the endsequent.  Then there is an $r^+$-derivation $d'$ of $\Theta$.
\end{theorem}
\begin{proof}
$d_?$ is an $r^+$-derivation of $(e,?,f),\Theta$.  By Lemma \ref{thm:f_to_t}, there is an $r^+$-deduction $d'_?$ of $(e,?,t),\Theta$ from sequents $(e,?,t),\Upsilon$ where $e$ is active.

Consider some such sequent in $d'_?$.  We may derive it by an $H_{e,u}$ inference from $\Sigma=(e,u),\Upsilon_{\leq r}$.  Since $\Sigma$ is the result of an $H$-step, $\Sigma$ is correct, by Lemma \ref{thm:pred_tail_char} we have $\Theta_{\leq r}t=\emptyset$, so by Lemma \ref{thm:pred_char}, we have $\Theta_{\leq r}\subseteq \Sigma$.  By Lemma \ref{thm:pred_rep}, $\Sigma\cup\Theta$ is defined, and there is a deduction of $\Sigma$ from $\Sigma\cup\Theta$.

$d_u$ is a deduction of $(e,u),\Theta$, and by Lemma \ref{thm:pred_char} we have $(\Sigma f)_{\geq r}\subseteq\Theta$ and, since $\Theta t\geq r$, also $\Sigma t\geq r$, so by Lemma \ref{thm:pred_compat}, $d_u\cup\Sigma$ is an $r^+$-derivation of $\Theta\cup\Sigma$.

So we have an $r^+$-derivation of each leaf of $d'_?$, giving an $r^+$-derivation $d''_?$ of $(e,?,t),\Theta$.  We obtain the desired deduction $d'$ by applying a $CutFr_e$ inference to $d''_?$ and the $d_u$.
\end{proof}

\begin{theorem}\label{thm:pred_reduction}
  If $d$ is an $r+1$-derivation of $\Theta$ with $r\neq\Omega$ and $\Theta$ has an $r+1$-path then there is an $r^+$-derivation $d'$ of $\Theta$.
\end{theorem}
\begin{proof}
  By induction on $d$.  Consider the last inference of $d$; if the last inference is anything other than a $Cut_e$ with $rk(e)=r$, the claim is immediate from the inductive hypothesis.

Suppose the last inference of $d$ is a $Cut_e$ with $rk(e)=r$.  We apply the inductive hypothesis to obtain $r^+$-deductions $d'_u$ of $(e,u),\Theta$.  We may then apply the previous theorem to obtain the desired $r^+$-derivation.
\end{proof}

We also need a slightly more complicated version of this argument when we have an $\Omega$-derivation.
\begin{theorem}\label{thm:omega_reduction}
If $d$ is an $\Omega$-derivation of $\Theta$ and $\Theta$ has a path consisting of $H$ and $Fr$ rules of rank $\geq\Omega$ and $Cut$ rules of rank $\leq r$ then there is an $r$-derivation $d'$ of $\Theta$.
\end{theorem}
\begin{proof}
  By induction on $d$.  Note that the path is an $s$-path for all $s>r$.  Consider the last inference of $d$; if the last inference is anything other than a $Cut_e$ with $rk(e)\geq r$, the claim is immediate from the inductive hypothesis.

Suppose the last inference of $d$ is a $Cut_e$ with $rk(e)\geq r$.  By the inductive hypothesis, there are $rk(e)+1$-derivations $d'_e$ of each $(e,u),\Theta$.  Combined with an application of $Cut_e$, we have an $rk(e)+1$-derivation of $\Theta$ with an $rk(e)+1$-path, so by $rk(e)-r+1$-applications of Theorem \ref{thm:pred_reduction} and Lemma \ref{thm:pred_prune}, there is an $r$-derivation of $\Theta$.
\end{proof}

\section{Impredicative Cut-Elimination}

\subsection{Outline of Elimination of rank $\Omega$ Cuts}

When $rk(e)\neq\Omega$, the formula $F(e,u)$ depends only on expressions of rank $<rk(e)$.  At rank $\Omega$, this breaks down, and the dependencies can even become circular.  In the method of local predicativity, this is resolved by stratifying the rank $\Omega$ expressions so that we can assign ranks which are countable ordinals.  As noted in the introduction, we avoid that approach here, but the idea that this should be possible motivates our construction.

Informally, we can identify which rank $\Omega$ expressions have an ``effective rank'' which is higher or lower than another expression.  By effective rank, we do not mean a formal assignment of ranks to expressions; rather, we mean that we can identify which expressions \emph{behave as if} they have lower rank than others in the context of a particular historical substitutions.  The essential behavior of rank is the following:
\begin{quote}
  $e$ has lower effective rank than $e'$ if and only if, in an $H$-step with main expression $e$, $e'$ is removed after the $H$-step.
\end{quote}
Our procedure for eliminating the rank $\Omega$ cuts is to arrange the process in an order which respects the effective rank.

The main complication in our setting is that the effective rank of $e$ is determined at the moment an $H$-step occurs---that is, we cannot expect to go through an $\Omega+1$-derivation and assign effective ranks to individual $Cut$ inferences; rather, we must assign the effective rank at $AxH$ axioms.

Given an $\Omega+1$-derivation of $\emptyset$, we first eliminate all $AxH_{c_n}$ axioms $\Theta$ where $(n\in I,\top)\not\in\Theta$; in such axioms, $c_n$ always has maximal effective rank (since if $(e,u)\in\Theta$ with $u\neq{?}$ and $rk(e)\leq\Omega$, then $(e,u)$ is also present in $H(\Theta)$).  Along the way we will convert all $Cut^{\Omega,\mathrm{term}}$ inferences into $CutFr^{\Omega,\mathrm{term}}$, but we will not prune them, since they will also be needed in the following step.

We describe what this process looks like; the main difficulties are caused by eliminating $AxH_{c_n}$ where $(c_m,?)$ is present for several other values of $c_m$.  In particular, different things can happen depending on whether the $Cut$ inference introducing $c_m$ occurs above or below the $Cut$ inference introducing $c_n$.  Our definitions have been rigged to make this go as smoothly as possible, as we now illustrate.  For simplicity, we illustrate the case with three $Cut^{\Omega,\mathrm{term}}$ inferences in immediate succession where we eliminate the middle one, since this demonstrates all the important features.

We begin with an $\Omega+1$ derivation like this:
  \begin{prooftree*}
\Infer0[$AxH_{c_n,u}$]{(c_k,?,f),(c_n,?,f),(c_m,?,f),\Sigma}
\Ellipsis{}{(c_k,?,f),(c_n,?,f),(c_m,?,f),\Theta}
\Hypo{\ldots}
\Infer2[$Cut_{c_m}$]{(c_k,?,f),(c_n,?,f),\Theta}
\Hypo{\ldots}
\Infer2[$Cut_{c_n}$]{(c_k,?,f),\Theta}
\Hypo{\ldots}
\Infer2[$Cut_{c_k}$]{\Theta}
\Ellipsis{}{\emptyset}

  \end{prooftree*}


After elimination, we get the derivation

\begin{prooftree*}
\Hypo{}
\Ellipsis{}{(c_k,?,t),(c_n,u,t),\Theta\cup\Sigma^*}
\Hypo{\ldots}
\Infer2[$CutFr_{c_k}$]{(c_n,u,t),\Theta\cup\Sigma^*}
\Ellipsis{}{(c_n,u,t),\Sigma^*}
\Infer1[$H_{c_n,u}$]{(c_k,?,t),(c_n,?,t),(c_m,?,t),\Sigma}
  \Ellipsis{}{(c_k,?,t),(c_n,?,t),(c_m,?,t),\Theta}
\Hypo{\ldots}
\Infer2[$CutFr_{c_m}$]{(c_k,?,t),(c_n,?,t),\Theta}
\Hypo{\ldots}
\Infer2[$CutFr_{c_n}$]{(c_k,?,t),\Theta}
\Hypo{\ldots}
\Infer2[$CutFr_{c_k}$]{\Theta}
\Ellipsis{}{\emptyset}
\end{prooftree*}

The main complication is that we need to insert a $CutFr_{c_k}$ when we carry out repetition.  In order to have the extra derivations needed to construct the side branches, we will use a more complicated inductive hypothesis.

We turn to the second stage, where we eliminate all remaining $AxH$ axioms of rank $\Omega$---that is, all the $AxH$ axioms where a formula is changed (either from negative to positive or positive to negative).  Here we need some additional work to ensure that we eliminate $AxH$ axioms in the right order.  The basic issue is this: suppose we have an $AxH_{n\in I}$ axiom $\Theta$; we expect to replace it with an $H$-inference.  As in predicative cut-elimination, we will have a derivation $d_\top$ which we intend to modify into a derivation of $H(\Theta)$.  The issue is that, in the derivation above $H(\Theta)$, all $H$-steps which keep $(n\in I,\top)$ have higher effective rank than $n\in I$, while all $H$-rules for $c_n$ have lower effective rank.  This means that when we place $d_\top$ above $H(\Theta)$, we should have already eliminated all $AxH$ axioms $\Sigma$ of rank $\Omega$ such that $n\in I$ is in $H(\Sigma)_P$; on the other hand, any $AxH_{c_m}$ axiom where $m=n$ or $m\in I$ appears before $n\in I$ in $\Theta_P$ should be eliminated later.

The simplest case goes as follows.  Suppose we have a derivation where this situation occurs.  (The order of the $CutFr$ and $Cut$ is not particularly significant.)
\begin{prooftree*}\tiny
  \Infer0[$AxH_{n\in I,\top}$]{(c_n,?,t),(n\in I,?,f),\Sigma}
  \Ellipsis{}{(c_n,?,t),(n\in I,?,f),\Theta}
  \Hypo{\ldots}
  \Hypo{\phantom{(c_n,u,f),(n\in I,?,f),\Theta}\tikzmark{q1}}
  \Ellipsis{}{(c_n,u,f),(n\in I,?,f),\Theta\tikzmark{q2}}
  \Infer3[$CutFr$]{(n\in I,?,f),\Theta}
  \Hypo{\ldots}
  \Hypo{\phantom{(n\in I,\top,f,P,V),\Theta}\tikzmark{q3}}
  \Ellipsis{}{(n\in I,\top,f,P,V),\Theta\tikzmark{q4}}
  \Infer3[$Cut$]{\Theta}
  \Ellipsis{}{\emptyset}
\end{prooftree*}

  \begin{tikzpicture}[remember picture, overlay]
  \draw[decorate, decoration = {brace, amplitude=2pt}] (pic cs:q1) -- ($ (pic cs:q2) + (0,.23) $) node[midway, right=3pt] {\tiny$d_u$};
  \draw[decorate, decoration = {brace, amplitude=2pt}] (pic cs:q3) -- ($ (pic cs:q4) + (0,.23) $) node[midway, right=3pt] {\tiny$d_{P,V}$};
  \end{tikzpicture}

As we proceed inductively, we first apply cut-elimination to the side-branch $d_{P,V}$; at the conclusion of this process we have an $\Omega$-deduction $d'_{P,V}$ of $(n\in I,\top,f,P,V),\Theta$ with the property that every $AxH$ axiom $\Upsilon$ of rank $\Omega$ has $H$-expression $c_m$ with either $m=n$ or $m\in I$ appearing before $n\in I$ in $P$.  (In particular, $(n\in I,\top)$ will not be present in $H(\Upsilon)$.)

We \emph{do not} apply cut-elimination to $d_u$ yet; instead we carry the unmodified derivation $d_u$ along as part of the inductive process.  When we eliminate this $AxH$ axiom, we place $d'_{P,V}$ above it (using the same process as in the predicative case---repeating the path to $\emptyset$ and then taking a union with $d'_{P,V}$).  We also modify $d'_{P,V}$ by replacing $(n\in I,\top,f)$ with $(n\in I,\top,t)$.  This causes certain $AxH$ axioms to be invalid---specifically, exactly those axioms with $c_n$ as the $H$-expression.

We do a further elimination step with these axioms, using the derivation $d_u$ which we carried along.  As always we perform the repetition step and \emph{now} we apply our process inductively to $d_u$ (while simultaneously performing the union step).  This is the correct order---we should apply cut-elimination to $d_u$ after we have placed $d'_{P,V}$ on top of $\Sigma$, since this respects the effective ranks---but it leads to a complicated inductive ordering which is specified in Section \ref{sec:wellf}.

\subsection{Leaving Formulas Alone}

Throughout this subsection, when the $H$-step applies to $(e,?),\Upsilon$ where $e=e((e,?),\Upsilon)$ and $rk(e)=\Omega$, we will write $\Upsilon^*=H((e,?),\Upsilon)$.

We need a notion intermediate between the $\Omega+1$ and $\Omega^+$-deduction:
\begin{definition}
We say $d$ is an $\Omega^{*}$-deduction if:
\begin{itemize}
\item $Cut(d)\leq\Omega, CutFr(d)=\Omega, Fr(d)>\Omega, H(d)\geq\Omega$,
\item There are no $Cut^{\Omega,\mathrm{term}}$ inferences in $d$.
\end{itemize}




\end{definition}

We introduce the following, basically ad hoc, notion to help us keep track of the rank $\Omega$ elements which remain fixed in $\Omega^{*}$-deductions:
\begin{definition}
  We write $\Theta s=\Theta_{=\Omega}^-\cup\Theta_{=\Omega}^{+,\mathrm{form}}\cup\{(c_m,?)\in\Theta\mid (m\in I,\top)\in\Theta\}$.
\end{definition}
Note that $\Theta s$ is the part of $\Theta_{=\Omega}$ which is kept after an $H$-step with $H$-expression $c_m$ where $m\in I$ is not in $\Theta_P$.

We need an analog of Lemma \ref{thm:pred_compat} suitable for this context:
\begin{lemma}\label{thm:imp_compat}
  Suppose $d$ is an $\Omega^{*}$-derivation of $\Theta$ and $\Sigma$ is a sequent such that:
  \begin{itemize}
  \item $\Sigma$ is correct,
  \item $\Theta$ and $\Sigma$ are compatible,
  \item $\Sigma \leq \Omega$,
  \item $\Sigma_{=\Omega}^{+,\mathrm{term}}\subseteq \Sigma s\cap \Sigma t$,
  \item $\Sigma t\geq\Omega$,
  \item $(\Theta\cup\Sigma)^{\mathrm{form}}_{=\Omega}\subseteq(\Theta\cup\Sigma)f$.
  \end{itemize}
Then there is an $\Omega^{*}$-derivation $d\cup\Sigma$ of $\Theta\cup\Sigma$.
\end{lemma}
\begin{proof}
  By induction on $d$.  Consider the final inference of $d$.  If the final inference is a $Cut_e$, a $Cut^{\Omega,\mathrm{form}}_e$, an $Fr_e$, or an $H_e$ with $rk(e)>\Omega$, we proceed as in Lemma \ref{thm:pred_compat}.

In the case of an $H_e$ inference with $rk(e)=\Omega$, we have to note that since there are no temporary formulas in $\Theta$, the $H$-expression must be a term $c_n$.  Since the $H$-rule applies to $\Theta$, $n\in I$ must be in $\dom(\Theta)$, and since $n\in I$ cannot be active, $(n\in I,?)\in\Theta_S$.  Therefore $(n\in I,?)\in(\Theta\cup\Sigma)_S$, so the only active expressions $\Theta\cup\Sigma$ are $(\Theta\cup\Sigma)^{+,\mathrm{term}}_{=\Omega}\setminus(\Theta\cup\Sigma)s$.  If $(c_n,?)\in(\Theta\cup\Sigma)^{+,\mathrm{term}}_{=\Omega}\setminus(\Theta\cup\Sigma)s$ then either $(c_n,?)\in\Theta^{+,\mathrm{term}}_{=\Omega}\setminus\Theta s$ (in which case $\Theta_F(c_n)=t$ because $c_n$ is active in $\Theta$) or $(c_n,?)\in\Sigma^{+,\mathrm{term}}_{=\Omega}\subseteq \Sigma t$.  We must check that the $H$-rule still applies; the only possible obstacle would be if $(c_m,?)\in\Theta\cup\Sigma$ but $m\in I\not\in\dom(\Theta\cup\Sigma)$.  But if $(c_m,?)\in\Theta$, $m\in I\in\dom(\Theta)\subseteq\dom(\Theta\cup\Sigma)$ because the $H$-rule applies to $\Theta$, and if $(c_m,?)\in\Sigma$ then $(c_m,?)\in\Sigma s$, so $(m\in I,\top)\in\Sigma$.  So we may derive $\Theta\cup\Sigma$ by an $H$-rule from $\Theta'\cup\Sigma$.  (Note that, unlike the predicative case, we do not delete elements from $\Sigma$ when we apply the inductive hypothesis.)


  Suppose the final inference is a $CutFr^{\Omega,\mathrm{term}}_e$.  If $e\in\dom(\Sigma)$ and $\Sigma_S(e)={?}$ then we must have $\Sigma_S(e)=t$ (since $(e,?)\not\in\Theta_S$), so we may take $d_{?}\cup\Sigma$.

Suppose $e\in\dom(\Sigma)$ but $\Sigma_S(e)\neq{?}$.  If $\Sigma_F(e)=f$, we may take $d_{\Sigma_S(e)}\cup\Sigma$.  Suppose $\Sigma_F(e)=t$; we apply the inductive hypothesis to obtain a deduction $d_{\Sigma_S(e)}\cup(\Sigma\setminus\{(e,\Sigma_S(e),t)\})$ of $(e,\Sigma_S(e),f),\Theta\cup(\Sigma\setminus\{(e,\Sigma_S(e),t)\})$.  By Lemma \ref{thm:f_to_t} we obtain a deduction $d'$ of $\Theta\cup\Sigma$ from sequents where $e$ is active.  Suppose $e$ is active at some sequent $\Upsilon$; then $e(\Upsilon)$ is a formula, and therefore must be fixed in $\Upsilon$ (since $\Theta\cup\Sigma$ contains no temporary formulas, and no inference rules in an $\Omega^{*}$-derivation can introduce temporary formulas), so $\Upsilon$ is an instance of $AxH$.  
 Then $d'$ is the desired derivation.

  If $e\not\in\dom(\Sigma)$, we apply the inductive hypothesis to obtain $\Omega^{*}$-derivations $d'_u$ of $(e,u,i),\Theta\cup\Sigma$ and apply a $CutFr^{\Omega,\mathrm{term}}_e$ inference to obtain the desired derivation $d'$ of $\Theta\cup\Sigma$.  
\end{proof}

\begin{definition}
  If $\Theta$ is a sequent, we write $\Theta^{\mathrm{term}\rightarrow t}$ for the sequent where:
  \begin{itemize}
  \item $\Theta^{\mathrm{term}\rightarrow t}_S=\Theta_S, \Theta^{\mathrm{term}\rightarrow t}_P=\Theta_P, \Theta^{\mathrm{term}\rightarrow t}_V=\Theta_V$, and
  \item $\Theta^{\mathrm{term}\rightarrow t}_F(e)=\left\{\begin{array}{ll}
t&\text{if }rk(e)=\Omega\text{, }e\text{ is a term, and }\Theta_S(e)={?}\\
\Theta_F(e)&\text{otherwise}
\end{array}\right.$
  \end{itemize}
\end{definition}

\begin{lemma}\label{thm:imp1_rep_compat}
Suppose $(\Theta_0,\ldots,\Theta_n)$ is an $\Omega+1$-path in an $\Omega+1$-derivation with $\Theta=\Theta_n$ and $\Sigma$ is a correct sequent such that:
\begin{itemize}
\item $\Sigma\leq\Omega$,
\item $\Theta^{\mathrm{term}\rightarrow t}_{\leq\Omega}\setminus\Theta^{\mathrm{term}\rightarrow t}t\subseteq\Sigma$,
\item $\Theta^{-,\mathrm{term}}_{=\Omega}\subseteq\Sigma^{-,\mathrm{term}}_{=\Omega}$,
\item if $(e,u)\in\Sigma^{-,\mathrm{term}}_{=\Omega}$ with $u\neq{?}$ then $(e,?)\not\in\Theta$,
\item $\Theta_P$ is an initial segment of $\Sigma_P$, and
\item for each $m\in I$ in $\Theta_P$, $\Theta_V(m\in I)=\Sigma_V(m\in I)$.
\end{itemize}

Then $\Theta^{\mathrm{term}\rightarrow t}$ and $\Sigma$ are compatible.
\end{lemma}
\begin{proof}
  By induction on $n$.  This is trivial when $n=0$.  Suppose $n>0$; let $\Theta'=\Theta_{n-1}$.  Since the inference between $\Theta_{n-1}$ and $\Theta$ is one occurring in an $\Omega+1$-derivation, the inductive hypothesis applies, so $\Theta_{n-1}$ and $\Sigma$ are compatible.  If this inference is anything other than a $Cut^{\Omega,\mathrm{term}}_e$ inference, we proceed as in Lemma \ref{thm:pred_rep}.

If this inference is a $Cut^{\Omega,\mathrm{term}}_e$ inference, we have $(e,u)\in\Theta$.  If $u\neq{?}$ then $(e,u)\in\Sigma$.  If $u={?}$ and $(e,u)\in\Sigma$ then since $(e,?)\in\Theta$, $u={?}$.  Otherwise $e\not\in\dom(\Sigma)$.
\end{proof}

\begin{lemma}\label{thm:imp1_rep}
Suppose $(\Theta_0,\ldots,\Theta_n)$ is an $\Omega+1$-path with $\Theta=\Theta_n$ and $\Sigma$ is a correct sequent such that:
\begin{itemize}
\item $\Sigma\leq\Omega$,
\item $\Theta^{\mathrm{term}\rightarrow t}_{\leq\Omega}\setminus\Theta^{\mathrm{term}\rightarrow t}t\subseteq\Sigma$,
\item $\Theta^{-,\mathrm{term}}_{=\Omega}\subseteq\Sigma^{-,\mathrm{term}}_{=\Omega}$,
\item if $(e,u)\in\Sigma^{-,\mathrm{term}}_{=\Omega}$ with $u\neq{?}$ then $(e,?)\not\in\Theta$,
\item $\Sigma^{+,\mathrm{term}}_{=\Omega}\subseteq\Sigma s\cap\Sigma t$,
\item $\Sigma t\geq\Omega$,
\item $\Sigma^{\mathrm{form}}_{=\Omega}\subseteq\Sigma f$,
\item $\Theta_P$ is an initial segment of $\Sigma_P$, and
\item for each $m\in I$ in $\Theta_P$, $\Theta_V(m\in I)=\Sigma_V(m\in I)$.
\end{itemize}

  Suppose that for each $i<n$ such that $\Theta_{i+1}=(e,?,f),\Theta_i$ where $rk(e)=\Omega$ and $e$ is a term, there is a family $\{d_{e,v}\}$ such that each $d_{e,v}$ is an $\Omega^{*}$-derivation of $(e,v,f),\Theta^{\mathrm{term}\rightarrow t}_i$.





Then there is a derivation $d$ of $\Sigma$ from $\Theta^{\mathrm{term}\rightarrow t}\cup\Sigma$ consisting only of $H$ and $Fr$ inferences of rank $>\Omega$, and $CutFr^{\Omega,\mathrm{term}}$ inferences.
\end{lemma}
\begin{proof}
By induction on $n$.  This is trivial when $n=0$.  Suppose $n>0$; let $\Theta'=\Theta_{n-1}$.  We have $\Theta'_{\leq \Omega}\subseteq\Theta_{\leq\Omega}$, so the inductive hypothesis applies to $\Theta'$, so there is a deduction of $\Sigma$ from $\Theta'\cup\Sigma$.  Consider the inference between $\Theta$ and $\Theta'$; if it is anything other than a $Cut^{\Omega,\mathrm{term}}_e$, we proceed as in Lemma \ref{thm:pred_rep}.

  Suppose the inference is a $Cut^{\Omega,\mathrm{term}}_e$, so some $(e,u)\in\Theta_S$.  If $u\neq {?}$ then $(e,u)\in\Sigma_S$ so $\Theta'\cup\Sigma=\Theta\cup\Sigma$.  If $u={?}$ then either $(e,u)\in\Sigma_S$, so $\Theta'\cup\Sigma=\Theta\cup\Sigma$, or $e\not\in\dom(\Sigma)$.  In the latter case, we derive $\Theta'\cup\Sigma$ by a $CutFr^{\Omega,\mathrm{term}}_e$ inference.  The $?$-branch is $\Theta\cup\Sigma$.

  To obtain the $u$-branches for $u\neq{?}$, observe that we have a derivation $d_{e,u}$ of $(e,u,f),\Theta^{\mathrm{term}\rightarrow t}_{n-1}$.   We attempt to apply Lemma \ref{thm:imp_compat} to obtain derivations $d_{e,u}\cup\Sigma$; by assumption $\Sigma$ is correct.  Since $\Theta^{\mathrm{term}\rightarrow t}$ and $\Sigma$ are compatible, $e\not\in\dom(\Sigma)$, $\Sigma\leq\Omega$, and $(\Theta_{n-1})_{\leq\Omega}\subseteq\Theta$, also $\Theta_{n-1}$ and $\Sigma$ are compatible.  By assumption, $\Sigma\leq\Omega$, $\Sigma^{+,\mathrm{term}}_{=\Omega}\subseteq\Sigma s\cap\Sigma t$, and $\Sigma t\geq \Omega$.  We also have $(\Theta_{n-1})^{\mathrm{form}}_{=\Omega}\subseteq\Theta_{n-1}f$ and $\Sigma^{\mathrm{form}}_{=\Omega}\subseteq\Sigma f$ by assumption.

Therefore a $CutFr^{\Omega,\mathrm{term}}_e$ inference applied to the derivations $d_{e,u}\cup\Sigma$ and $\Theta\cup\Sigma$ gives the desired derivation.
\end{proof}

\begin{lemma}\label{thm:omega_star_red}
  Suppose $d$ is an $\Omega+1$-derivation of $\Theta$ such that there is an $\Omega+1$-path $\emptyset=\Theta_0,\ldots,\Theta_n=\Theta$ for $\Theta$.

  Suppose that for each $i<n$ such that $\Theta_{i+1}=(e,?,f),\Theta_i$ where $rk(e)=\Omega$ and $e$ is a term, there is a family $\{d_{e,v}\}$ such that each $d_{e,v}$ is an $\Omega^{*}$-derivation of $(e,v,f),\Theta^{\mathrm{term}\rightarrow t}_i$.

Then there is an $\Omega^{*}$-derivation of $\Theta^{\mathrm{term}\rightarrow t}$.
\end{lemma}
\begin{proof}
  By induction on $d$.  We consider the final inference of $d$.  If $\Theta$ is an axiom other than $AxH$, $\Theta^{\mathrm{term}\rightarrow t}$ is an instance of the same axiom.  

Suppose the last inference is anything other than a $Cut^{\Omega,\mathrm{term}}_e$ or an $AxH_{e,v}$, so we infer $\Theta$ from $\{\Theta_i\}_{i\in\mathcal{I}}$ with an inference $\mathrm{I}$ where, since $\mathrm{I}$ appears in an $\Omega+1$-derivation, $\mathrm{I}$ is an $H_e$ or $Fr_e$ with $rk(e)>\Omega$, a $Cut_e$ with $rk(e)<\Omega$, or a $Cut^{\Omega,\mathrm{form}}_e$ inference.  
  So the inductive hypothesis applies to each $d_i$ to obtain $\Omega^{*}$-deductions of $\Theta^{\mathrm{term}\rightarrow t}_i$, and applying the same inference $\mathrm{I}$ gives the desired deduction of $\Theta^{\mathrm{term}\rightarrow t}$.

  Suppose the last inference is a $Cut^{\Omega,\mathrm{term}}_e$.  Then there are $\Omega+1$-derivations $d_u$ of $(e,u,f),\Theta$ for each $u\in\mathbb{N}$.  We may apply the inductive hypothesis to each side branch to obtain $\Omega^{*}$-deductions $d'_u$ of $(e,u,f),\Theta^{\mathrm{term}\rightarrow t}$.  We then apply the inductive hypothesis to the deduction $d_?$ of $(e,?,f),\Theta$ taking $d_{e,u}=d'_u$.  We combine these inferences with a $CutFr^{\Omega,\mathrm{term}}_e$. 

  Suppose the last inference is an $AxH_{e,v}$ axiom and whenever $e'$ is active in $\Theta^{\mathrm{term}\rightarrow t}$, $\Theta^{\mathrm{term}\rightarrow t}_F(e')=t$.  Since $\Theta$ is an instance of $AxH_{e,v}$, there must be some active $e'$ which is a term of rank $\Omega$.  Since no formula is temporary in $\Theta^{\mathrm{term}\rightarrow t}$, $e$ must be a term $c_n$ with $(n\in I,\top)\not\in\Theta_S$.  We may deduce $\Theta^{\mathrm{term}\rightarrow t}$ from a sequent $\Sigma=(e,v,t),\Theta^*$ by an $H_{e,v}$ inference.

First, we collect some facts about $\Sigma$: $\Sigma\leq\Omega$ since it is the premise of an $H$-rule of rank $\Omega$.  Since this was an $H$-rule for a term, also $\Sigma^{+,\mathrm{term}}_{=\Omega}\subseteq\Sigma s$; if $(e,?)\in\Sigma^{+,\mathrm{term}}_{=\Omega}$ then $(e,?)\in\Theta^{\mathrm{term}\rightarrow t}$, so $(e,?)\in\Sigma t$.  Since there is an $\Omega+1$-path for $\Theta$, $\Sigma t\geq\Omega$ and $\Sigma^{\mathrm{form}}_{=\Omega}\subseteq\Sigma f$.

  We have an $\Omega^*$-derivation $d_{e,v}$ of $(e,v,f),\Theta_i$ for some $i$.  We construct a derivation of $\Sigma$ from $\Theta^{\mathrm{term}\rightarrow t}_i\cup\Sigma$ using Lemma \ref{thm:imp1_rep}.  There is an $\Omega+1$-path for $\Theta_i$.  Note that $\Theta^{\mathrm{term}\rightarrow t}_{\leq\Omega}\setminus\Theta^{\mathrm{term}\rightarrow t} t\subseteq\Sigma$, so $(\Theta^{\mathrm{term}\rightarrow t}_i)_{\leq\Omega}\setminus\Theta^{\mathrm{term}\rightarrow t}_it\subseteq\Sigma$.  We have $(\Theta_i)^{-,\mathrm{term}}_{=\Omega}\subseteq\Theta^{-,\mathrm{term}}_{=\Omega}\subseteq\Sigma^{-,\mathrm{term}}_{=\Omega}$.  If $(e',?)\in(\Theta_i)^{+,\mathrm{term}}_{=\Omega}$ then $(e',?)\in\Theta$ and $e'\neq e$, so if $e'\in\dom(\Sigma)$ then $(e',?)\in\Sigma$.  Various properties of $\Sigma$ are checked above.  $(\Theta_i)_P$ is an initial segment of $\Theta_P=\Sigma_P$, and for each $m\in I$ in $\Theta_P$ we have $\Theta_V(m\in I)=\Sigma_V(m\in I)$.  By Lemma \ref{thm:imp1_rep_compat}, $\Sigma$ and $\Theta_i^{\mathrm{term}\rightarrow t}$ are compatible.  So Lemma \ref{thm:imp1_rep} applies, so $\Theta^{\mathrm{term}\rightarrow t}_i$ and $\Sigma$ are compatible and there is a deduction of $\Sigma$ from $\Theta^{\mathrm{term}\rightarrow t}_i\cup\Sigma$ consisting only of $H$ and $Fr$ inferences of rank $>\Omega$ and $CutFr^{\Omega,\mathrm{term}}$ inferences.

  It remains to find an $\Omega^{*}$-deduction of $\Theta^{\mathrm{term}\rightarrow t}_i\cup\Sigma$.  We have shown that $(e,v,f),\Theta^{\mathrm{term}\rightarrow t}_i$ and $\Theta^*$ are compatible, $\Theta^*\leq\Omega$, $(\Theta^*)^{+,\mathrm{term}}_{=\Omega}\subseteq\Theta^*s\cap\Theta^*t$, $\Theta^*t\geq\Omega$, and all formulas are fixed in both $\Theta^*$ and $\Theta_i$, so we may apply Lemma \ref{thm:imp_compat} to obtain an $\Omega^{*}$-deduction of $(e,v,f),\Theta^{\mathrm{term}\rightarrow t}_i\cup\Theta^*$.  By Lemma \ref{thm:f_to_t}, we obtain a deduction $d^*$ of $\Theta_i\cup\Sigma$ from sequents where $e$ is active.  But if $e$ is active at some $\Upsilon$ in $d^*$, the $H$-expression must be a formula, and all formulas are fixed in $\Upsilon$, so $\Upsilon$ is an $AxH$ axiom.  So $d^*$ is an $\Omega^{*}$-derivation.

Combining $d^*$, the deduction of $\Sigma$ from $\Theta^{\mathrm{term}\rightarrow t}_i\cup\Sigma$, and an $H_{e,v}$ rule gives an $\Omega^{*}$-derivation of $\Theta^{\mathrm{term}\rightarrow t}$.
\end{proof}

\subsection{Well-Foundedness}\label{sec:wellf}

Before proving the main theorem, we need to establish the induction on derivations which we will need.  We will consider triples $(d,R,\{d_{p,u}\}_{p\in R,u\in\mathbb{N}})$ where $R$ is (for the moment) an arbitrary index set.  For notational convenience, we assume $\top$ is a fixed element not in $R$, and we will take $d_{\top,u}=d$ for all $u$.
\begin{definition}
  We say $(d',R',\{d'_{p,u}\}_{u\in\mathbb{N},p\in R'}) \prec^* (d,R,\{d_{p,u}\}_{u\in\mathbb{N},p\in R})$ if there is a function $f:R'\cup\{\top\}\rightarrow R\cup\{\top\}$ such that:
  \begin{itemize}
  \item each $d_{p',u}$ is a subderivation of $d_{f(p'),u}$,
  \item if $f^{-1}(f(p'))$ is not a singleton or $f(p')=\top$ then each $d_{p',u}$ is a proper subderivation of $d_{f(p'),u}$.
  \end{itemize}
\end{definition}

\begin{lemma}
  $\prec^*$ is transitive.
\end{lemma}
\begin{proof}
  If $(d'',R'',\{d''_{p,u}\}_{u\in\mathbb{N},p\in R''})  \prec^*(d',R',\{d'_{p,u}\}_{u\in\mathbb{N},p\in R'}) \prec^* (d,R,\{d_{p,u}\}_{u\in\mathbb{N},p\in R})$, there are functions $f':R''\cup\{\top\}\rightarrow R'\cup\{\top\}$ and $f:R'\cup\{\top\}\rightarrow R\cup\{\top\}$ witnessing this.  We claim that $f\circ f'$ witnesses that $(d'',R'',\{d''_{p,u}\}_{u\in\mathbb{N},p\in R''}) \prec^* (d,R,\{d_{p,u}\}_{u\in\mathbb{N},p\in R})$: for any $p''$, we have $d_{p'',u}$ is a subderivation of $d_{f'(p''),u}$, which in turn is a subderivation of $d_{f(f'(p'')),u}$ as needed.  To check the second clause, certainly if $f(f'(p''))=\top$ then $d'_{f'(p''),u}$ is a proper subderivation of $d_{\top,u}$, soit suffices to show that if $f(f'(p''_0))=f(f'(p''_1))$ then $d_{p''_0,u}$ is a proper subderivation of $d_{f(f'(p''_0)),u}$.  If $f'(p''_0)=f'(p''_1)$ then $d_{p''_0,u}$ is a proper subderivation of $d_{f'(p''_0),u}$, and so a proper subderivation of $d_{f(f'(p''_0)),u}$.  Otherwise $f'(p''_0)\neq f'(p''_1)$ but $f(f'(p''_0))=f(f'(p''_1))$, so $d_{f'(p''_0),u}$ is a proper subderivation of $d_{f(f'(p''_0)),u}$.
\end{proof}

\begin{lemma}
  $\prec^*$ is well-founded.
\end{lemma}
\begin{proof}
  Suppose $(d^0,R^0,\{d^0_{p,u}\}_{u\in\mathbb{N},p\in R^0})\succ^* (d^1,R^1,\{d^1_{p,u}\}_{u\in\mathbb{N},p\in R^1})\succ^*\cdots$.  For $j<i$ we have $f_{i,j}:R^i\cup\{\top\}\rightarrow R^j\cup\{\top\}$ witnessing $(d^j,R^j,\{d^j_{p,u}\}_{u\in\mathbb{N},p\in R^j})\prec^* (d^i,R^i,\{d^i_{p,u}\}_{u\in\mathbb{N},p\in R^i})$ with $f_{i,j}\circ f_{j,k}=f_{i,k}$.  

We say $p'\in R^i\cup\{\top\}$ is a \emph{child} of $p\in R^j\cup\{\top\}$ if $f_{i,j}(p')=p$.  We say that $p\in R^j$ \emph{shrinks} at $i$ if for every child $p'$ of $p$ in $R^i$, each $d^i_{p',u}$ is a proper subderivation of $d^j_{f(p'),u}$.  Clearly if $p$ shrinks at $R^i$, $p$ also shrinks at all $i'\geq i$.

Choose $p_0\in R^0\cup\{\top\}$ such that $f_{i,0}(\top)=p_0$ for infinitely many $i$.  Choose an $i_0$ so that $f_{i_0,0}(\top)=p_0$.  We claim that $p_0$ shrinks at $i_0+1$: if $p_0$ does not already shrink at $i_0$ then $\top$ is the only child of $p_0$ in $i_0$, so $p_0$ must shrink in all $i'> i_0$.  We may repeat this process: choose a $p_1$ in $R^{i_0+1}\cup\{\top\}$ with $f_{i_0+1,0}(p_1)=p_0$ and such that $f_{i,i_0+1}(\top)=p_1$ for infinitely many $i$.  Then, taking $i_1$ to be least so that $f_{i_1,i_0+1}(\top)=p_1$, $p_1$ must shrink at $i_1+1$.  Iterating this process, we obtain an infinite sequence of properly decreasing derivations $d^0_{p_0,0}, d^{i_0+1}_{p_1,0},\ldots$, contradicting the well-foundedness of $d^0_{p_0,0}$.
\end{proof}

\subsection{Eliminating Formulas}

\begin{lemma}\label{thm:imp2_compat}
  Suppose $d$ is an $\Omega$-derivation of $\Theta$.  Let $\Sigma$ be a correct sequent such that:
  \begin{itemize}
  \item $\Sigma\leq\Omega$,
  \item $\Theta$ and $\Sigma$ are compatible,
  \item $(\Sigma f)_{\geq\Omega}\subseteq\Theta$,
  \item $\Sigma^-_{=\Omega}=\emptyset$,
  \item $\Sigma_P$ is an initial segment of $\Theta_P$.
  \item for each $m\in I$ in $\Sigma_P$, $\Theta_V(m\in I)=\Sigma_V(m\in I)$.
  \end{itemize}
Then there is an $\Omega$-derivation $d\cup\Sigma$ of $\Theta\cup\Sigma$.
\end{lemma}
\begin{proof}
  By induction on $d$.  Consider the final inference of $d$.  A $Cut$ (necessarily of rank $<\Omega$), or an $Fr$ or $H$ of rank $>\Omega$ is handled as in Lemma \ref{thm:pred_compat}.

  At an $Fr_e$ of rank $\Omega$, if $e$ is a formula, we cannot have $(e,?)\in\Sigma$ since $\Sigma^-_{=\Omega}=\emptyset$, and we cannot have $(e,\top)\in\Sigma$ since $\Sigma_P\sqsubseteq\Theta_P$; by the inductive hypothesis we have an $\Omega$-derivation of $(e,?,t),\Theta\cup\Sigma$, and an application of the same inference gives a deduction of $\Theta\cup\Sigma$.

If $e$ is a term, if $e\not\in\dom(\Sigma)$ then again by the inductive hypothesis we have an $\Omega$-derivation of $(e,?,t),\Theta\cup\Sigma$ and an application of the same inference gives an $\Omega$-derivation of $\Theta\cup\Sigma$.  If $e\in\dom(\Sigma)$ then, since $\Sigma^-_{=\Omega}=\emptyset$, $(e,?)\in\Sigma$, so we may apply the inductive hypothesis to obtain an $\Omega$-derivation of $(e,?,t),\Theta\cup\Sigma=\Theta\cup\Sigma$.

  At an $H_e$ of rank $\Omega$, note that the $H$-rule also applies to $\Theta\cup\Sigma$ (since $(\Sigma f)_{\geq\Omega}\subseteq\Theta$).  We may derive $\Theta\cup\Sigma$ from some sequent $\Upsilon$ using an $H_e$ inference; letting $\Sigma'=\Sigma\cap\Upsilon$, the inductive hypothesis gives us an $\Omega$-derivation of $\Upsilon$.
\end{proof}

\begin{definition}
    If $\Theta$ is a sequent, we write $\Theta^{\mathrm{form}\rightarrow t}$ for the sequent where:
  \begin{itemize}
  \item $\Theta^{\mathrm{form}\rightarrow t}_S=\Theta_S, \Theta^{\mathrm{form}\rightarrow t}_P=\Theta_P, \Theta^{\mathrm{form}\rightarrow t}_V(m\in I)=\Theta_V$, and
  \item $\Theta^{\mathrm{form}\rightarrow t}_F(e)=\left\{\begin{array}{ll}
t&\text{if }rk(e)=\Omega\text{, }e\text{ is a formula, and }\Theta_S(e)=?\\
\Theta_F(e)&\text{otherwise}
\end{array}\right.$
  \end{itemize}
\end{definition}

\begin{lemma}\label{thm:imp2_rep}
  Suppose $(\Theta_0,\ldots,\Theta_n)$ is an $\Omega^{*}$-path with $\Theta=\Theta_n$ and $\Sigma$ is a correct sequence such that:
  \begin{itemize}
  \item whenever $(e,u)\in\Theta_{=\Omega}^{-,\mathrm{term}}$ then $\Theta_F(e)=t$,
  \item whenever $(e,?)\in\Sigma_{=\Omega}^{+,\mathrm{term}}$ then $\Sigma_F(e)=t$,
  \item $\Sigma\leq\Omega$,
  \item $\Theta_{\leq\Omega}^{\mathrm{form}\rightarrow t}\setminus\Theta^{\mathrm{form}\rightarrow t} t\subseteq\Sigma$,
  \item if $(c_n,u)\in\Theta^{-,\mathrm{term}}_{=\Omega}$ then $(n\in I,\top),(c_n,?)\not\in\Sigma$,
  \item $\Theta_P$ is an initial segment of $\Sigma_P$,
  \item for each $m\in I$ in $\Theta_P$, $\Theta_V(m\in I)=\Sigma_V(m\in I)$.
  \end{itemize}

Then $\Theta^{\mathrm{form}\rightarrow t}$ and $\Sigma$ are compatible and there is a deduction of $\Sigma$ from $\Theta^{\mathrm{form}\rightarrow t}\cup\Sigma$ consisting only of $Fr$ and $H$ inferences of rank $\geq\Omega$.
\end{lemma}
\begin{proof}
  By induction on $n$.  This is trivial when $n=0$.  Suppose $n>0$; let $\Theta'=\Theta_{n-1}$.  The inference rule between $\Theta_n$ and $\Theta_{n-1}$ is one which can occur in an $\Omega^{*}$-deduction, so $\Theta'_{\leq\Omega}\setminus \Theta' t\subseteq\Theta_{\leq\Omega}\setminus \Theta t$ and if $(c_n,u)\in\Theta'$ with $u\neq {?}$ then $(c_n,u)\in\Theta$, so the inductive hypothesis applies to $\Theta'$.

Then there is a deduction of $\Sigma$ from $(\Theta')^{\mathrm{form}\rightarrow t}\cup\Sigma$.  Consider the inference between $\Theta$ and $\Theta'$:
  \begin{itemize}
  \item $Cut_e$, $Fr_e$: as in Lemma \ref{thm:pred_rep},
  \item $Cut^{\Omega,\mathrm{form}}_e$: if $(e,\top)\in\Theta_S$ then $(e,\top)\in\Sigma_S$, so $(\Theta')^{\mathrm{form}\rightarrow t}\cup\Sigma=\Theta^{\mathrm{form}\rightarrow t}\cup\Sigma$ and we are done.  If $(e,?)\in\Theta_S$ then either $(e,?)\in\Sigma$, in which case again $(\Theta')^{\mathrm{form}\rightarrow t}\cup\Sigma=\Theta^{\mathrm{form}\rightarrow t}\cup\Sigma$, or $e\not\in\dom(\Sigma)$, so $(\Theta')^{\mathrm{form}\rightarrow t}\cup\Sigma$ may be derived from $\Theta^{\mathrm{form}\rightarrow t}\cup\Sigma$ by an $Fr$ inference.
  \item $H_{e,v}$: if this is an $H$-rule of rank $\Omega$ with $H$-expression $c_n$ then we have $(c_n,u)\in\Theta^{-,\mathrm{term}}_{=\Omega}$, so $(n\in I,\top),(c_n,?)\not\in\Sigma$; therefore as in Lemma \ref{thm:pred_rep} there is a derivation of $(\Theta')^{\mathrm{form}\rightarrow t}\cup\Sigma$ from a sequent $\Upsilon$ by an $H_{e,v}$ inference.  We have $\Upsilon\subseteq \Theta^{\mathrm{form}\rightarrow t}\cup\Sigma$ but there may be $(c_m,?)\in\Sigma_S\setminus\Upsilon$.  By applying an $Fr_{c_m}$ for each such $c_m$, we obtain a derivation of $\Upsilon$ from $\Theta^{\mathrm{form}\rightarrow t}\cup\Sigma$.
  \item $CutFr^{\Omega,\mathrm{term}}_e$: we cannot be a side branch, so we apply an $Fr_e$ inference if $(e,?)\not\in\Sigma$ and do nothing if $(e,?)\in\Sigma$.
  \end{itemize}
\end{proof}

\begin{definition}


  We say $\Gamma$ is \emph{permitted} for $\Theta$ if:
  \begin{itemize}
  \item $(\Gamma f)^{-,\mathrm{term}}_{=\Omega}=\emptyset$,
  \item there is an $\Omega^*$-path for $\Gamma$,
  \item $\Gamma_P\sqsubseteq\Theta_P$,
  \item $\Gamma_{\leq\Omega}\setminus\Gamma t\subseteq\Theta$, and
  \item $\Gamma^{-,\mathrm{term}}_{=\Omega}\subseteq\Theta$.
  \end{itemize}
\end{definition}



\begin{theorem}\label{thm:impred_reduction}
Suppose $d$ is an $\Omega^{*}$-derivation of $\Theta$ and that there is an $\Omega^{*}$-path for $\Theta$.  Suppose $(\Theta f)^{-,\mathrm{term}}_{=\Omega}=\emptyset$. 

Let $\Lambda$ be a sequent with $\Lambda\leq\Omega$, $\Lambda t\geq \Omega$, $\Lambda$ is compatible with $\Theta^{\mathrm{form}\rightarrow t}$, $\Theta_P\sqsubseteq\Lambda_P$, if $n\in I$ is in $\dom(\Theta_P)$ then $\Theta_V(n\in I)=\Lambda_V(n\in I)$, $(\Lambda f)^{-,\mathrm{term}}_{=\Omega}=(\Lambda f)^{+,\mathrm{term}}_{=\Omega}=\emptyset$, and if $(c_n,?)\in\Lambda$ then $(n\in I,\top,f)\in\Lambda$. 

Suppose there is a set $R$ of pairs $(c_n,\Gamma)$ and, for each $p=(c_n,\Gamma)\in R$, a collection of $\Omega^{*}$-derivations $\{d_{p,u}\}_{u\in\mathbb{N}}$ such that:
\begin{itemize}
\item each $d_{p,u}$ is an $\Omega^{*}$-derivation of $(e,u,f),\Gamma$,
\item if $(c_n,?)\in\Theta^{\mathrm{form}\rightarrow t}$ and $(n\in I,\top)\not\in\Theta^{\mathrm{form}\rightarrow t}\cup\Lambda$ then there is a pair $(c_n,\Gamma)\in R$ with $\Gamma$ permitted for $\Theta$,
\item if $(e,\Delta)\in R$ and $(c_n,?)\in\Delta$ then there is a pair $(c_n,\Gamma)$ in $R$ with $\Gamma$ permitted for $\Delta$.
\end{itemize}

Finally, suppose that for each formula $e$ with $(e,?)\in\Theta$, there is a sequent $\Gamma_e$ and a collection $\{d_{e,P,V}\}$ such that:
\begin{itemize}
\item each $d_{e,P,V}$ is an $\Omega$-derivation of $(e,\top,f,P,V),\Gamma^{\mathrm{form}\rightarrow t}_e$,
\item there is an $\Omega^{*}$-path for $\Gamma_e$,
\item $(\Gamma_e)_{\leq\Omega}\setminus\Gamma_et\subseteq\Theta$,
\item $(\Gamma_ef)^{-,\mathrm{term}}_{=\Omega}=\emptyset$,
\item $(\Gamma_e)^{-,\mathrm{term}}_{=\Omega}\subseteq\Theta$,
\item $(\Gamma_e)_P$ is an initial segment of $\Theta_P$,
\item whenever $(c_n,u)\in\Gamma_e$ with $u\neq{?}$, $(n\in I,\top),(c_n,?)\not\in\Theta$.
\end{itemize}

Then there is an $\Omega$-derivation of $\Theta^{\mathrm{form}\rightarrow t}\cup\Lambda$.
\end{theorem}
\begin{proof}
We proceed by induction on the ordering $\prec^*$ on $(d,R,\{d_{p,u}\}_{p\in R,u\in\mathbb{N}})$.  We refer to $R,\{d_{p,u}\}_{p\in R,u\in\mathbb{N}}$ as the \emph{side premises}; when we refer to applying the inductive hypothesis to a subderivation $d'$ of $d$, we mean using the inductive hypothesis for $(d',R,\{d_{p,u}\}_{p\in R,u\in\mathbb{N}})$.  The derivations $d_{e,P,V}$ will be called the \emph{extra data} to distinguish them.

We will split into cases based on the final inference of $d$ as usual.  With each step, we have to check that we have suitable side premises; this means both keeping track of which side premises are needed (that is, how big $R$ needs to be) and checking that the needed side premises remain valid.

Consider the final inference of $d$.  If this a $Cut_e$ inference with $e\in\dom(\Lambda)$, we have $rk(e)<\Omega$, so $(e,u,f)\in\Lambda$ for some $u$, and the claim follows by applying the inductive hypothesis to $d_u$.  If $e\not\in\dom(\Lambda)$, we may apply the inductive hypothesis to each $d_u$ and then apply a $Cut_e$ rule.  Note that none of these cases interfere with the side-premises.

If the last inference is an $Fr_e$ inference, so $rk(e)>\Omega$, the inductive hypothesis gives a derivation of $(e,?,t),\Theta^{\mathrm{form}\rightarrow t}\cup\Lambda$, and an $Fr_e$ inference gives a derivation of $\Theta^{\mathrm{form}\rightarrow t}\cup\Lambda$.  
Again, this case does not interfere with the side-premises.

Suppose the last inference is a $CutFr^{\Omega,\mathrm{term}}_e$ inference, so we have an $\Omega^*$-derivation of $(e,?,t),\Theta$ and for each $u\neq{?}$ we have an $\Omega^{*}$-derivation of $(e,u,f),\Theta$.  If either $e\not\in\dom(\Lambda)$ or $(e,?,t)\in\Lambda$, we apply the inductive hypothesis to $(d_?,R\cup\{(e,\Theta)\},\{d_{p,u}\}_{p\in R,u\in\mathbb{N}}\cup\{d_u\})$, since $\Theta$ is clearly permitted for $(e,?,t),\Theta$.  This gives us a deduction of $(e,?,t),\Theta^{\mathrm{form}\rightarrow t}\cup\Lambda$, and we apply either apply an $Fr_e$ inference (if $e\not\in\dom(\Lambda)$) or simply take this deduction (if $(e,?,t)\in\Lambda$) to obtain a deduction of $\Theta^{\mathrm{form}\rightarrow t}\cup\Lambda$.

  If $(e,u,t)\in\Lambda$ then we apply Lemma \ref{thm:f_to_t} to $d_u$ to obtain a deduction of $d'_u$ from sequences where $u$ is active; since all such sequents must have $H$-expression a formula, and all formulas are fixed in $d'_u$, $d'_u$ is a derivation of $(e,u,t),\Theta$.  Then we apply the inductive hypothesis to $d'_u$.  


  Suppose the last inference is a $Cut^{\Omega,\mathrm{form}}_e$, so we have side branches $d_{P,V}$ which are $\Omega^{*}$-derivations of $(e,\top,f,P,V),\Theta$.  We apply the inductive hypothesis with $\Lambda=\emptyset$ to each side branch to obtain $\Omega$-derivations of $(e,\top,f,P,V),\Theta^{\mathrm{form}\rightarrow t}$.  We use these as the additional extra data and apply the inductive hypothesis to the derivation of $(e,?,f),\Theta$, so we obtain a derivation of $(e,?,t),\Theta^{\mathrm{form}\rightarrow t}\cup\Lambda$.  We then apply an $Fr$ inference to obtain a derivation of $\Theta^{\mathrm{form}\rightarrow t}\cup\Lambda$.

Suppose the last inference is an $H_e$ inference.  If $rk(e)>\Omega$, this follows from the inductive hypothesis.  If $rk(e)=\Omega$, we have an $\Omega^{*}$-derivation $d'$ of the premise $\Theta^*$ and $e$ must be a term.  We attempt to apply the inductive hypothesis with an appropriate subset of $\Lambda$ to obtain a derivation of $(\Theta^{\mathrm{form}\rightarrow t}\cup\Lambda)^*$.  If $(c_n,?)\in(\Theta^{\mathrm{form}\rightarrow t}\cup\Lambda)^*$, it must be because $(n\in I,\top)\in\Theta\cup\Lambda$, so we have the needed side-premises.

Suppose $\Theta$ is an instance of an axiom but $\Theta^{\mathrm{form}\rightarrow t}\cup\Lambda$ is not.  Then $\Theta$ must be an instance of $AxH_{e,v}$ with $rk(e)=\Omega$ and $e$ a formula.  We may derive $\Theta^{\mathrm{form}\rightarrow t}\cup\Lambda$ from a sequent $\Sigma=(e,\top,t,P,V),\Theta^*$ by an $H_{e,\top}$-inference rule.  Since $(e,?)\in\Theta$, there is an $\Omega$-derivation $d_{e,P,V}$ of $(e,\top,f,P,V),\Gamma^{\mathrm{form}\rightarrow t}_e$.  We attempt to apply Lemma \ref{thm:imp2_rep}: if $(e',u)\in(\Gamma_e)^{-,\mathrm{term}}_{=\Omega}$ then $(\Gamma_e)_F(e')=t$ by assumption, and if $(e,?)\in\Sigma^{\mathrm{term}}_{=\Omega}$ then $\Sigma_F(e)=t$ by assumption.  $\Sigma\leq\Omega$ since it is the result of an $H$-step of rank $\Omega$.  By assumption, $(\Gamma_e)_{\leq\Omega}\setminus\Gamma_e t\subseteq \Theta$, and since $\Theta^{\mathrm{form}\rightarrow t}_{\leq\Omega}\setminus\Theta^{\mathrm{form}\rightarrow t} t\subseteq\Sigma$, we have $(\Gamma_e)^{\mathrm{form}\rightarrow t}_{\leq\Omega}\setminus\Gamma^{\mathrm{form}\rightarrow t}_e t\subseteq\Sigma$, and $\Gamma_P$ is an initial segment of $\Sigma_P$.  If $(c_n,u)\in\Gamma_e$ with $u\neq{?}$ then $(c_n,u)\in\Theta$, so $(n\in I,\top),(c_n,?)\not\in\Sigma$.  So there is a deduction of $\Sigma$ from $\Gamma_e^{\mathrm{form}\rightarrow t}\cup\Sigma$.

It remains to produce an $\Omega$-derivation of $(\Gamma_e\cup\Sigma)^{\mathrm{form}\rightarrow t}=\Gamma_e^{\mathrm{form}\rightarrow t}\cup\Sigma=(e,\top,t,P,V),\Gamma_e^{\mathrm{form}\rightarrow t}\cup\Theta^*$.  We have $\Theta^*\leq\Omega$, and $\Theta^*$ and $(e,\top,f,P,V),\Gamma_e$ are compatible.  Since $\Theta^*$ is the result of an $H$-step with a formula $H$-expression, $(\Theta^*)^-_{=\Omega}=\emptyset$, $(\Theta^* f)_{\geq\Omega}\subseteq(\Theta^*)^{+,\mathrm{form}}_{=\Omega}\subseteq((e,\top,f,P,V),\Gamma_e)^{+,\mathrm{form}}_{=\Omega}$, and $\Theta^*_P\sqsubseteq ((e,\top,t,P,V),\Gamma_e)_P$.  Therefore by Lemma \ref{thm:imp2_compat}, there is an $\Omega$-derivation of $(e,\top,f,P,V),\Gamma_e\cup\Theta^*$.  By Lemma \ref{thm:f_to_t}, we have a deduction of $(e,\top,t,P,V),\Gamma^{\mathrm{form}\rightarrow t}_e\cup\Theta^*$ from sequents where $e$ is active. 

We now show that we can produce an $\Omega$-derivation of each of these sequents.  Consider some sequent $\Upsilon$ in this deduction so that the $H$-step applies to $\Upsilon$ and $e$ is active.  Then the $H$-expression of $\Upsilon$ must be a term of the form $c_n$ where $n\in I$ appears in $((e,\top,t,P,V),\Gamma_e)_P$---in particular, if $e$ is not $n\in I$ then $n\in I$ appears before $e$.  But since there is an $\Omega^{*}$-path for $\Gamma_e$, any $n\in I$ appearing in $P$ strictly before $e$ is fixed, so if $n\in I$ appears before $e$, $\Upsilon$ is an instance of $AxH$.

So $e$ is $n\in I$ and the $H$-expression of $\Upsilon$ is $c_n$.  We may deduce $\Upsilon$ by an $H$-inference from $\Upsilon'=(c_n,u,t),\Upsilon^*$.  Since the $H$-step applies to $\Theta$, we have $(c_n,?)\in\Theta$, and therefore we have $(c_n,\Gamma_{c_n})\in R$ and a family of deductions $d_{c_n,u}$ of $(c_n,u,f),\Gamma_{c_n}$ so that $\Gamma_{c_n}$ is permitted for $\Theta$.

We apply Lemma \ref{thm:imp2_rep} (with $\Gamma_{c_n}$ as $\Theta$ and $\Upsilon'$ as $\Sigma$).  If $(e',u)\in(\Gamma_{c_n})^{-,\mathrm{term}}_{=\Omega}$ then since $(\Gamma_{c_n}f)^{-,\mathrm{term}}_{=\Omega}=\emptyset$, $(\Gamma_{c_n})_F(e')=t$.  If $(e',?)\in(\Upsilon')_{=\Omega}^{+,\mathrm{term}}$ then $\Upsilon'_F(e')=t$ because of the path for $\Upsilon$.  $\Upsilon'\leq\Omega$ since it is the result of an $H$-step of rank $\Omega$.  If $(e',u)\in(\Gamma_{c_n}^{\mathrm{form}\rightarrow t})_{\leq\Omega}\setminus(\Gamma_{c_n}^{\mathrm{form}\rightarrow t}t)$ then $(e',u)$ is fixed in $\Theta$, and since the inferences between $\Upsilon'$ and $\Theta$ all come from an $\Omega$-derivation, also $(e',u)\in\Upsilon'$.  If $(c_m,u)\in(\Gamma_{c_n})^{-,\mathrm{term}}_{=\Omega}$ then $(c_m,u)\in\Theta$, so $(c_m,u)\in\Upsilon_V(n\in I)$, so $(m\in I,\top),(c_m,?)\not\in\Upsilon'$.  If $(m\in I,\top)\in\Gamma_{c_n}$ then since $(\Gamma_{c_n})_P\sqsubseteq\Theta_P=\Upsilon'_P$, $(m\in I,\top)\in\Upsilon'$.  Therefore $\Gamma_{c_n}^{\mathrm{form}\rightarrow t}$ and $\Upsilon'$ are compatible and there is a deduction of $\Upsilon'$ from $\Gamma_{c_n}^{\mathrm{form}\rightarrow t}\cup\Upsilon'$.

We apply Lemma \ref{thm:f_to_t} to $d_{c_n,u}$ to obtain an $\Omega^*$-deduction $d'_{c_n,u}$ of $(c_n,u,t),\Gamma_{c_n}$ from sequents where $c_n$ is active; since all such sequents must have $H$-expression a formula, and all formulas are fixed in $\Gamma_{c_n}$, $d'_{c_n,u}$ is also an $\Omega^*$-derivation.  Finally, we apply the inductive hypothesis to $d'_{c_n,u}$ with $\Lambda=\Upsilon^*$ and the needed side-premises.  Clearly $((c_n,u,t),\Gamma_{c_n}f)^{-,\mathrm{term}}_{=\Omega}=(\Upsilon^*)^{-,\mathrm{term}}_{=\Omega}=\emptyset$, $\Upsilon^*\leq\Omega$, $\Upsilon^*t\geq\Omega$, and we have shown that $\Upsilon^*$ is compatible with $\Gamma_{c_n}^{\mathrm{form}\rightarrow t}$.  Because $\Upsilon^*$ is the result of an $H$-step with $H$-expression $c_n$, if $(c_m,?)\in\Upsilon^*$ then $(m\in I,\top,f)\in\Upsilon^*$.
\end{proof}

\subsection{Putting it Together}

\begin{theorem}
  Suppose $ID_1\vdash\exists x\phi(x)$ where $\phi$ is quantifier-free and does not contain the symbol $I$.  Then there is an $n$ such that $\phi(n)$ holds.
\end{theorem}

Since this is a consistency result, one usually considers explicitly what system the proof takes place in.  The usual approach would be to describe a system of ordinal notations and add, to each of the steps in the cut-elimination proof, a calculation showing that that the steps increase the ordinal bounds on the proofs in a controlled way.  Since the ordinal notation systems for $ID_1$ are both well-known and complicated \cite{MR2457679,MR0036806,MR0329869}, and since assigning ordinal heights to the combinatorial steps described above is routine, we omit the calculation of ordinals from this already along paper.  The reader who is troubled by a theorem claiming a consistency proof may prefer to assume that all proofs in this paper are formalized in $\mathrm{RCA}_0+WO(\alpha)$, and therefore that the main theorem says
\begin{quote}
  The theory $\mathrm{RCA}_0+WO(\alpha)$ proves that whenever $ID_1\vdash\exists x\phi(x)$ where $\phi$ is quantifier-free and does not contain the symbol $I$, there is an $n$ such that $\phi(n)$ holds.
\end{quote}
Here $\mathrm{RCA}_0$ is a second order theory conservative over primitive recursive arithmetic \cite{simpson99} (the second order part is being used to formalize the infinitary proof system; this could also be done in primitive recursive arithmetic with extra coding) and $\alpha$ is a system of ordinal notations for a sufficiently large computable ordinal.  (This ordinal $\alpha$ should be the Howard-Bachmann ordinal; without doing the necessary calculations, we cannot claim to have given a new proof that this specific ordinal suffices, however.)

\begin{proof}
  Since $ID_1\vdash\exists x\phi(x)$, also $ID_1^\epsilon\vdash\exists x\phi(x)$ (interpreting this as the formula $\phi(c_{\exists x\,\phi(x,\vec y)}(\vec d),\vec d)$ where $\vec d$ is all closed terms in $\phi$.

  Then by Theorem \ref{thm:initial_deriv}, there is a derivation of $\emptyset$ consisting of $Cut$-inferences of rank $\leq\Omega+r$ for some finite $r$.  In particular, this is an $\Omega+r$-derivation of $\emptyset$.  By $r-1$ applications of Theorem \ref{thm:pred_reduction} followed by Lemma \ref{thm:pred_prune}, there is an $\Omega+1$-derivation of $\emptyset$.

By Lemma \ref{thm:omega_star_red}, there is an $\Omega^*$-derivation of $\emptyset$ with no unjustified terms.   We apply Theorem \ref{thm:impred_reduction} with $\Lambda=\emptyset$ and $R=\emptyset$ to obtain an $\Omega$-derivation of $\emptyset$.

Therefore by Theorem \ref{thm:omega_reduction}, there is a $0$-derivation of $\emptyset$.  But in a $0$-derivation, the only possible inferences are $Fr$ and $H$ inferences.  Therefore this derivation is a line, and the sequents are $\emptyset=\Theta_0,\Theta_1,\ldots,\Theta_k$ for some $k$, and $\Theta_k$ must be an axiom.  Each $\Theta_i$ must be correct since $Fr$ and $H$ inferences preserve correctness, so $\Theta_k$ is not an instance of $AxF$.  $\Theta_k$ cannot be an instance of $AxH$ since $(\Theta_k)_F$ must be constantly equal to $t$.  So $\Theta_k$ is an instance of $AxS$, and so is a solving $\epsilon$-substitution.

In particular, $\Theta_k\vDash\exists x\phi(x)$, and therefore $\phi(\Theta_k(c_{\exists x\,\phi(x,\vec y)}(\vec d)),\vec d)$ must be true.
\end{proof}

\bibliographystyle{plain}
\bibliography{../../Bibliographies/main}
\end{document}